%% file: scr2.tex
\documentclass[11pt,letterpaper]{article}
\input{sections/scr2_preamble}
\usepackage[margin=1in]{geometry}

\makeatletter
\def\blfootnote{\gdef\@thefnmark{}\@footnotetext}
\makeatother

\begin{document}

	\title{Shifted Composition II: Shift Harnack Inequalities \\ and Curvature Upper Bounds}

 	\author{
		Jason M. Altschuler\\
		UPenn\\
		\texttt{alts@upenn.edu}
		\and
		Sinho Chewi \\
		IAS \\
		\texttt{schewi@ias.edu}
	}
	\date{}
	\maketitle

	\begin{abstract}
            We apply the shifted composition rule---an information-theoretic principle introduced in our earlier work~\cite{scr1}---to establish shift Harnack inequalities for the Langevin diffusion.
            We obtain \emph{sharp} constants for these inequalities for the first time, allowing us to investigate their relationship with other properties of the diffusion.
            Namely, we show that they are equivalent to a sharp ``local gradient-entropy'' bound, and that they imply curvature \emph{upper} bounds in a compelling reflection of the Bakry{--}\'{E}mery theory of curvature \emph{lower} bounds.
            Finally, we show that the local gradient-entropy inequality implies optimal concentration of the score, a.k.a.\ the logarithmic gradient of the density. 
	\end{abstract}

    \small
	\setcounter{tocdepth}{2}
	\tableofcontents	
	\normalsize

    \input{sections/intro}
    \input{sections/prelim}
    \input{sections/discrete}
    \input{sections/continuous}
    \input{sections/harnack}

	\paragraph*{Acknowledgements.} JMA acknowledges the support of an NYU Faculty Fellowship. SC acknowledges the support of the Eric and Wendy Schmidt Fund at the Institute for Advanced Study. 

    \newpage

    \appendix
    \input{sections/app}

    \small
	\addcontentsline{toc}{section}{References}
    \printbibliography{}
    
\end{document}

%% file: sections/scr2_preamble.tex
\usepackage[utf8]{inputenc}

\usepackage{amsmath, amsthm}
\usepackage{MnSymbol} 
\usepackage{graphicx,color}
\usepackage[hyphens]{url}
\usepackage{dsfont}
\usepackage{booktabs}
\usepackage[normalem]{ulem}
\usepackage{mathtools}
\usepackage{hyperref}
\usepackage[nameinlink,capitalize]{cleveref}
\usepackage{multirow}
\usepackage{algorithm}
\usepackage[noend]{algpseudocode}
\usepackage{tikz} %

\usepackage{soul} %

\usepackage[giveninits=true, maxnames=10, style=alphabetic, sorting=nyt]{biblatex}
\usepackage{subcaption}

\usepackage{tabularx}
\newcolumntype{L}{>{\raggedright\arraybackslash}X}

\DeclareMathAlphabet\mathbb{U}{msb}{m}{n}

\numberwithin{equation}{section}
\newtheorem{theorem}{Theorem}[section]
\newtheorem{cor}[theorem]{Corollary}
\newtheorem{lemma}[theorem]{Lemma}
\newtheorem{remark}[theorem]{Remark}

\newtheorem{defin}[theorem]{Definition}

\newcommand{\cD}{\mathcal{D}}

\newcommand{\cN}{\mathcal{N}}

\newcommand{\R}{\mathbb{R}}

\newcommand{\N}{\mathbb{N}}

\newcommand*{\E}{\mathbb{E}}

\renewcommand{\le}{\leqslant}
\renewcommand{\ge}{\geqslant}
\renewcommand{\leq}{\leqslant}
\renewcommand{\geq}{\geqslant}

\providecommand{\abs}[1]{\lvert{#1}\rvert}

\providecommand{\norm}[1]{\lVert{#1}\rVert}

\newcommand{\Hell}{\mathsf{D}}

\newcommand{\KL}{\mathsf{KL}}

\newcommand{\Ren}{\mathsf{R}}

\usepackage{bbm}
\usepackage{blkarray}
\usepackage{cancel}
\usepackage{enumitem}
\usepackage{float}
\usepackage{mathrsfs}
\usepackage[framemethod=tikz]{mdframed}
\usepackage{microtype}
\usepackage{ragged2e}
\usepackage{sectsty}
\usepackage{thmtools}
\usepackage{thm-restate}
\usepackage[Symbolsmallscale]{upgreek}
\usepackage{array}

\usetikzlibrary{cd}

\hypersetup{citecolor=red, colorlinks=true, linkcolor=blue}

\makeatletter
\patchcmd{\@counteralias}
{\@ifdefinable{c@#1}}
{\expandafter\@ifdefinable\csname c@#1\endcsname}
{}{}
\makeatother

\makeatletter
\newcommand{\opnorm}{\@ifstar\@opnorms\@opnorm}
\newcommand{\@opnorms}[1]{%
	\left|\mkern-1.5mu\left|\mkern-1.5mu\left|
	#1
	\right|\mkern-1.5mu\right|\mkern-1.5mu\right|
}
\newcommand{\@opnorm}[2][]{%
	\mathopen{#1|\mkern-1.5mu#1|\mkern-1.5mu#1|}
	#2
	\mathclose{#1|\mkern-1.5mu#1|\mkern-1.5mu#1|}
}
\makeatother

\newcommand{\PreserveBackslash}[1]{\let\temp=\\#1\let\\=\temp}
\newcolumntype{C}[1]{>{\PreserveBackslash\centering}p{#1}}

\newcommand\bs[1]{\boldsymbol{#1}}

\newcommand\mc[1]{\mathcal{#1}}

\newcommand\ms[1]{\mathscr{#1}}

\newcommand\msf[1]{\mathsf{#1}}

\definecolor{MITBrown}{RGB}{164, 31, 50}

\DeclareMathOperator\divergence{div}

\DeclareMathOperator\ent{ent}

\DeclareMathOperator\law{law}

\DeclareMathOperator\var{var}

\newcommand{\D}{\mathrm{d}}

\newcommand{\Coup}{\ms C}

\newcommand\deq{\coloneqq}

\newcommand\mmid{\mathbin{\|}}

\newcommand\T{\mathsf{T}}

\newcommand{\esssup}[1]{\mathop{#1\text{-}\mathrm{ess\,sup}}}

\newcommand{\CD}{\msf{CD}}

\newcounter{dummy}
\makeatletter
\newcommand\myitem[1][]{\item[#1]\refstepcounter{dummy}\def\@currentlabel{#1}}
\makeatother

\tikzset{
    shadedNode/.style={rectangle, draw=none, fill=blue!20, inner sep=1mm}
}

%% file: sections/intro.tex

\section{Introduction}\label{sec:intro}

This paper is a sequel to our earlier work~\cite{scr1}, in which we introduced an information-theoretic principle called the \emph{shifted composition rule}. Briefly, this principle allows for bounding information-theoretic divergences, such as the Kullback{--}Leibler (KL) or R\'enyi divergence, between the marginal laws of two stochastic processes through the introduction of a third, auxiliary process. In that paper, we applied the shifted composition rule to provide, among other results, information-theoretic proofs of F.-Y.\ Wang's celebrated dimension-free Harnack inequalities for diffusions~\cite{Wang1997LSINoncompact} via their dual formulation as a certain family of reverse transport inequalities.

The Harnack inequalities therein encode regularity for Kolmogorov's \emph{backward} equation. The aim of the present paper is to apply our information-theoretic framework to the dual problem of regularity for Kolmogorov's \emph{forward} equation, i.e., the Fokker{--}Planck equation. To describe the problem setting and our results, we introduce some basic concepts. 

\paragraph*{Harnack and shift Harnack inequalities.}
For concreteness, let $V : \R^d\to\R$ be a smooth potential and consider the Langevin diffusion
\begin{align}\label{eq:langevin}
    \D X_t = -\nabla V(X_t) \, \D t + \sqrt 2 \,\D B_t\,,
\end{align}
where ${(B_t)}_{t\ge 0}$ is a standard Brownian motion on $\R^d$.
Following the storied tradition of functional analysis, we study the regularity of this process through its \emph{Markov semigroup} ${(P_t)}_{t\ge 0}$, which acts on any (reasonable) function $f : \R^d\to\R$ via $P_t f(x) \deq \E[f(X_t) \mid X_0 = x]$.
Classically, under a curvature lower bound of the form $\nabla^2 V \succeq \alpha I$, $\alpha \in \R$ (also called the \emph{curvature-dimension condition}, or the \emph{Bakry{--}\'Emery criterion}), one obtains Harnack inequalities of the form
\begin{align}\label{eq:intro_harnack}
    {P_t f(x)}^p \le C(p,t,x,y)\, P_t(f^p)(y)\,, \qquad \text{for all}~ x,y\in\R^d\,,\, t > 0\,,~\text{and}~f : \R^d\to\R_{\geq 0}\,.
\end{align}
Here, $p > 1$ and $C(p,t,x,y) > 0$ is an appropriate constant.
Such inequalities witness the regularizing effect of the semigroup: they imply that for every $t > 0$, the semigroup $P_t$ maps bounded functions into differentiable ones.
If the semigroup is described by transition densities $(t,x,y)\mapsto p_t(x,y)$, then this amounts to regularity properties with respect to the \emph{backward} variable $x$, i.e., with respect to perturbations to the initial condition of the diffusion.
Moreover, the sharp Harnack inequalities imply back the curvature lower bound $\nabla^2 V \succeq \alpha I$; see~\cite{Wang04EquivHarnack, Wang10HarnackBoundary, bakry2014analysis, Wang14Diffusion} or~\cite[\S 6]{scr1} for further discussion.

To address the regularity of Kolmogorov's \emph{forward} equation, the seminal work of F.-Y.\ Wang in~\cite{Wang14ShiftHarnack} introduced the family of \emph{shift Harnack inequalities}, written
\begin{align}\label{eq:intro_shift_harnack}
    {P_t(f(\cdot + v))}^p \le C(p, t, v)\, P_t(f^p)\,, \qquad \text{for all}~v\in\R^d,\, t > 0\,,~\text{and}~f : \R^d\to\R_{\geq 0}\,.
\end{align}
Shift Harnack inequalities imply, for example, the existence of the transition densities ${(p_t)}_{t\ge 0}$ with respect to the Lebesgue measure, and they also entail gradient bounds for the Lebesgue density $p_t(x,y)$ with respect to the \emph{forward} variable $y$ (a.k.a.\ the terminal condition for the diffusion). See \S\ref{sec:shift_harnack_bg} for further background and discussion.
Note that the Harnack~\eqref{eq:intro_harnack} and shift Harnack~\eqref{eq:intro_shift_harnack} inequalities differ because the Langevin semigroup does not commute with convolutions.

In his original paper~\cite{Wang14ShiftHarnack}, F.-Y.\ Wang established a number of applications and equivalences for shift Harnack inequalities. (See also his monographs~\cite{Wang13HarnackSPDE, Wang14Diffusion} for comprehensive expositions and further applications.) Through the use of coupling arguments, he then established integration by parts formulas and shift Harnack inequalities for (degenerate) stochastic Hamiltonian systems and some SPDEs.
Subsequent work extended his techniques to an abundance of settings, including SDEs driven by fractional Brownian motion~\cite{Fan15FBM}, SDEs driven by jump processes~\cite{Wang16IBPJumps}, other examples of SPDEs~\cite{Zhang16ShiftHarnack, LvHua21HarnackSPDEs}, McKean{--}Vlasov equations~\cite{Wang18Landau, HuaWan19DistDepSingular}, and SDEs with irregular coefficients~\cite{Huang19HarnackIntegrable, HuaWan19DistDepSingular, LvHua21HarnackSPDEs}.

Our starting point is the equivalent reformulation of the shift Harnack inequality~\eqref{eq:intro_shift_harnack}, via H\"older duality, in the form of a reverse transport inequality:
\begin{align}\label{eq:intro_rev_transport}
    \Ren_q(\delta_x P_t * \delta_v \mmid \delta_x P_t)
    &\le C'(p, t, v)\,, \qquad\text{for all}~x\in\R^d\,, t > 0\,,
\end{align}
see \S\ref{ssec:prelim:duality}.
Here, $\Ren_q$ is the R\'enyi divergence of order $q \deq \tfrac{p}{p-1}$ (see \S\ref{ssec:prelim:info} for information-theoretic preliminaries).
We apply the shifted composition rule to provide information-theoretic arguments, in discrete time, to establish~\eqref{eq:intro_rev_transport} and hence~\eqref{eq:intro_shift_harnack}.
In this respect, our development parallels our earlier work~\cite{scr1} at least at a syntactic level, and we organize our paper accordingly to emphasize these similarities.
However, despite the syntactic similarity, the problem of forward regularity is substantially different from that of backward regularity, and the latter is far less well-understood.

\paragraph*{Relationship with curvature upper bounds.}
Indeed, whereas F.-Y.\ Wang's original Harnack inequalities and the celebrated web of equivalences around the curvature-dimension condition have been well-understood for at least two decades, the understanding of shift Harnack inequalities is relatively nascent.\footnote{See, for example,~\cite{Wang14ShiftHarnack} for a discussion of the greater challenges faced in the forward regularity setting.}
A significant point of departure is that the optimal constants $C(p,t,v)$ in~\eqref{eq:intro_shift_harnack} are not sharply characterized.
In fact, the bounds established in the literature involve constants $C(p,t,v)$ which \emph{diverge to infinity} as $t\to\infty$.
Consequently, such bounds do not yield meaningful information about the regularity of the stationary distribution.

One of the primary contributions of this paper is to prove shift Harnack inequalities with \emph{optimal} constants.
With these sharp inequalities in hand, we are then in a position to investigate the possibility of developing \emph{equivalences} for the shift Harnack inequalities.
Towards this end, we prove the following chain of implications.

\begin{restatable}{theorem}{equivthm}\label{thm:equiv}
    Let ${(P_t)}_{t\ge 0}$ denote the Markov semigroup corresponding to the Langevin diffusion with potential $V$.
    Let $\beta > 0$ and $p, q > 1$.
    Consider the following properties.
    \begin{enumerate}[leftmargin=5.5em]
        \myitem[$(\msf{CurvBdd})$]\label{CurvBDD} The two-sided curvature bound $-\beta I \preceq \nabla^2 V \preceq \beta I$ holds.
        \myitem[$(\msf{LGE})$]\label{LGE} The local gradient-entropy bound
        \begin{align*}
            \frac{\norm{P_t \nabla f}^2}{P_t f} \le \frac{2\beta}{1-\exp(-2\beta t)}\,\{P_t(f\log f) - P_t f \log P_t f\}
        \end{align*}
        holds for all $t > 0$ and all smooth $f : \R^d\to \R_{>0}$.
        \myitem[$(\msf{SH}_p)$]\label{SHp} The shift Harnack inequality
        \begin{align*}
            {P_t(f(\cdot + v))}^p
            &\le \exp\Bigl( \frac{\beta p\,\norm v^2}{2\,(p-1)\,(1-\exp(-2\beta t))}\Bigr)\, P_t(f^p)
        \end{align*}
        holds for all $v\in\R^d$, all $t > 0$, and all $f : \R^d\to\R_{>0}$.
        \myitem[$(\msf{SRT}_q)$]\label{SRTq} The shift reverse transport inequality
        \begin{align*}
            \Ren_q(\delta_x P_t * \delta_v \mmid \delta_x P_t)
            &\le \frac{\beta q\,\norm v^2}{2\,(1-\exp(-2\beta t))}
        \end{align*}
        holds for all $x,v\in\R^d$ and all $t > 0$.
        \myitem[$(\msf{SH}_{\rm log})$]\label{SHlog} The shift log-Harnack inequality
        \begin{align*}
            P_t(f(\cdot + v))
            &\le \log P_t(\exp f) + \frac{\beta \,\norm v^2}{2\,(1-\exp(-2\beta t))}
        \end{align*}
        holds for all $v\in\R^d$, all $t > 0$, and all $f : \R^d\to\R_{>0}$.
        \myitem[$(\msf{SRT}_1)$]\label{SRTlog} The shift reverse transport inequality
        \begin{align*}
            \KL(\delta_x P_t * \delta_v \mmid \delta_x P_t)
            &\le \frac{\beta\,\norm v^2}{2\,(1-\exp(-2\beta t))}
        \end{align*}
        holds for all $x,v\in\R^d$ and all $t > 0$.
        \myitem[$(\msf{CurvUB})$]\label{CurvUB} The curvature upper bound $\nabla^2 V \preceq \beta I$ holds.
    \end{enumerate}
    Then, the following implications hold.
    \vspace{0.1cm}
\begin{center}
\begin{mdframed}[backgroundcolor=blue!5]
    \begin{tikzcd}[
        ampersand replacement=\&,
        background color=blue!5,
        column sep=1em,
        row sep=1em,
        every arrow/.append style={-Latex, line width=0.3mm, black}
    ]
        \text{\ref{CurvBDD}} \arrow[r, Rightarrow] \& \text{\ref{LGE}} \arrow[r, Leftrightarrow] \& (\forall p > 1~\text{\ref{SHp}}) \arrow[r, Rightarrow] \& (\exists p > 1~\text{\ref{SHp}}) \arrow[r, Rightarrow] \& \text{\ref{SHlog}} \arrow[r, Rightarrow] \& \text{\ref{CurvUB}} \\
        \& \& (\forall q > 1~\text{\ref{SRTq}}) \arrow[r, Rightarrow]\arrow[u, Leftrightarrow] \& (\exists q > 1~\text{\ref{SRTq}}) \arrow[r, Rightarrow]\arrow[u, Leftrightarrow] \& \text{\ref{SRTlog}} \arrow[u, Leftrightarrow]
    \end{tikzcd}
\end{mdframed}
\end{center}
\end{restatable}
 \vspace{0.1cm}
The implications~\ref{LGE} $\Leftrightarrow$~\ref{SHp} and~\ref{SHp} $\Rightarrow$~\ref{SHlog} are known from~\cite{Wang14ShiftHarnack} and~\cite[Theorem 1.3.5]{Wang13HarnackSPDE} respectively, and the equivalences~\ref{SHp} $\Leftrightarrow$~\ref{SRTq} (for the dual exponent $q = \frac{p}{p-1}$) and~\ref{SHlog} $\Leftrightarrow$~\ref{SRTlog} follow from duality (see \S\ref{ssec:prelim:duality}).
However, to our knowledge, the sharp forms of all the functional inequalities~\ref{LGE},~\ref{SHp},~\ref{SHlog},~\ref{SRTq}, and~\ref{SRTlog} appear here for the first time.
Among these implications, we found it surprising that the shift Harnack inequalities imply back the curvature upper bound, in a remarkable parallel to the equivalences with the curvature-dimension condition which captures curvature \emph{lower} bounds.
It is a tantalizing question to close this chain of implications, e.g., by showing that~\ref{CurvUB} is sufficient to imply~\ref{LGE} or~\ref{SHp}.
A related question is to characterize~\ref{CurvUB} via a gradient commutation inequality.
We discuss these directions further in \S\ref{sec:equiv}.

The motivation for these questions stems from the corresponding synthetic characterizations of Ricci curvature lower bounds, which extended the reach of tools from geometric analysis to more general metric measure spaces~\cite{Sturm06MMSI, Sturm06MMSII, LotVil09Ric, AmbGigSav14MMS, AmbGigSav15BakryEmery, ErbKuwStu15Bochner}.
In contrast, we are only aware of the works~\cite{Wu20UpperBakryEmery, Sturm21UpperRicci} on characterizations of Ricci curvature upper bounds.

\paragraph*{Implications for the stationary distribution.}
Suppose now that the Langevin diffusion~\eqref{eq:langevin} admits a unique stationary distribution $\pi$, so that necessarily $\pi\propto\exp(-V)$.
This holds true as soon as $\int\exp(-V) < \infty$.

Since the sharp constants in Theorem~\ref{thm:equiv} do not diverge as $t \to \infty$, our results immediately imply functional inequalities for the stationary distribution $\pi$ which may be of independent interest (see Corollary~\ref{cor:stationary}).
As one example,~\ref{LGE} reads
\begin{align}\label{eq:LGE_pi}\tag{$\msf{LGE}_\pi$}
    \frac{\norm{\E_\pi\nabla f}^2}{\E_\pi f}
    &\le 2\beta\ent_\pi(f) \qquad\text{for all}~f : \R^d\to\R_{>0}\,.
\end{align}
We show that this inequality implies a sharp concentration bound for $\nabla V$---often called the \emph{score} function in the statistical literature---under $\pi$. To illustrate this result, suppose that $\pi = \cN(0, \beta^{-1} I)$ with $\beta > 0$, so that $\nabla V(x) = \beta x$.
Then, it is easy to see that under $\pi$, $\nabla V$ has law $\cN(0, \beta I)$ and, in particular, is a $\sqrt\beta$-sub-Gaussian random vector. (We recall that a random vector $X$ is called \emph{$\sigma$-sub-Gaussian} if, for every unit vector $e\in\R^d$, $\langle e, X\rangle$ is $\sigma^2$-sub-Gaussian, i.e., $\E\exp(\lambda \,\langle e, X\rangle) \le \exp(\lambda^2 \sigma^2/2)$ for all $\lambda \in \R$.)
Our concentration inequality shows that this observation extends far beyond the Gaussian case.

\begin{restatable}{theorem}{gradVbd}\label{thm:gradV_concentration}
    Let $\pi\propto\exp(-V)$ be a probability measure on $\R^d$ such that~\eqref{eq:LGE_pi} holds.
    Then, under $\pi$, the score $\nabla V$ is $\sqrt\beta$-sub-Gaussian.
\end{restatable}

Note that a classical integration-by-parts identity shows that the condition $\nabla^2 V \preceq \beta I$ implies $\E_\pi[\norm{\nabla V}^2] = \E_\pi[\Delta V] \le \beta d$.
A concentration bound for $\norm{\nabla V}$ was previously derived in~\cite{leeshentian2020gradientconcentration} in the context of log-concave sampling, under the additional structural assumption of \emph{convexity} of $V$, and with subexponential tails rather than sub-Gaussian.
In contrast, our argument requires only the smoothness of $V$ (more precisely, through~\eqref{eq:LGE_pi}) and recovers the correct behavior for the Gaussian.
See \S\ref{sec:implications_stationary} for further discussion.

\paragraph*{Organization.}
We begin in \S\ref{sec:prelim} by collecting relevant preliminary material on information theory and in particular on the shifted composition rule (Theorem~\ref{thm:scr}).
In \S\ref{sec:disc}, we apply the shifted composition rule to establish sharp shift Harnack inequalities for discretized diffusions, from which the corresponding results for the continuous-time diffusion follow immediately as a limiting case.
In \S\ref{sec:cont}, we provide a direct continuous-time perspective that yields two alternative derivations of these inequalities, one via stochastic calculus and the other via direct differentiation of the R\'enyi divergence. The former is a refinement of the original argument of~\cite{Wang14ShiftHarnack} that achieves optimal constants; the latter is new. Finally, in \S\ref{sec:harnack}, we provide background on shift Harnack inequalities and duality, and we prove Theorems~\ref{thm:equiv} and~\ref{thm:gradV_concentration}.

\paragraph*{Setting.}
For the reverse transport inequalities in \S\ref{sec:disc} and \S\ref{sec:cont}, we work with It\^o diffusions $\{X_t\}_{t\ge 0}$ of the form
\begin{align*}
    \D X_t = b(X_t) \, \D t + \sigma \, \D B_t\,, \qquad\, X_0 = x\,,
\end{align*}
in the uniformly elliptic setting that $\sigma \sigma^\T \succeq \lambda I$ for some $\lambda > 0$. We assume that $b$ is $L$-Lipschitz; this implies in particular that there is a unique strong solution to the SDE\@. This level of generality comes at no cost because the arguments are not made any more difficult. The results for the Langevin diffusion follow immediately as a special case.

%% file: sections/prelim.tex

\section{Preliminaries}\label{sec:prelim}

Here we collect background material on information theory (\S\ref{ssec:prelim:info}), the shifted composition rule (\S\ref{ssec:scr}), and the convexity principle (\S\ref{ssec:prelim:convexity}). The latter two concepts were developed in~\cite{scr1}.

\subsection{Information theory}\label{ssec:prelim:info}

As discussed in \S\ref{sec:intro}, shift Harnack inequalities are equivalent to reverse transport inequalities involving R\'enyi divergences. These divergences are defined as follows.

\begin{defin}[R\'enyi divergence]
    Let $q \ge 1$. The \emph{R\'enyi divergence} of order $q$ between probability measures $\mu$, $\nu$ is defined to be
    \begin{align}\label{eq:renyi}
        \Ren_q(\mu \mmid \nu)
        &\deq \frac{1}{q-1} \log \int \bigl(\frac{\D\mu}{\D\nu}\bigr)^q \, \D \nu\,,
    \end{align}
    if $\mu \ll \nu$, and is defined to be $\infty$ otherwise.
\end{defin}

We make use of the following relations between R\'enyi divergences and other information divergences:
\begin{itemize}
    \item \underline{KL divergence.} For $q=1$, the definition~\eqref{eq:renyi} is interpreted in the limiting sense and equals the Kullback--Leibler (KL) divergence
    \begin{align*}
        \KL(\mu \mmid \nu)
        \deq \Ren_1(\mu \mmid \nu)
        \deq \int \Bigl(\frac{\D \mu}{\D\nu} \log \frac{\D\mu}{\D\nu}\Bigr) \, \D \nu\,.
    \end{align*}
    \item \underline{Chi-squared divergence.} For $q=2$, the R\'enyi divergence is related to the chi-squared divergence
    \begin{align*}
        \chi^2(\mu\mmid \nu)
        &\deq \var_\nu \frac{\D\mu}{\D\nu}
        = \int \Bigl(\frac{\D\mu}{\D\nu}\Bigr)^2 \, \D \nu - 1\,,
    \end{align*}
    via the expression $\Ren_2(\mu\mmid \nu) = \exp(1+\chi^2(\mu \mmid \nu))$.
    \item \underline{$f$-divergences.} For $q > 1$, the R\'enyi divergence is related to the $f$-divergence
    \begin{align}\label{eq:Hell_q}
        \Hell_q(\mu \mmid \nu)
        &\deq \int \bigl( \frac{\D\mu}{\D\nu}\bigr)^q \, \D \nu - 1
    \end{align}
    via the expression $\Ren_q(\mu \mmid \nu) = \frac{1}{q-1} 
        \log(1+\Hell_q(\mu\mmid \nu))$.
\end{itemize}

We refer the reader to the surveys~\cite{van2014renyi,mironov2017renyi} for further discussion of R\'enyi divergences and their properties. For the convenience of the reader, the following theorem collects the basic properties of R\'enyi divergences that we employ in this paper. 

\begin{theorem}\label{thm:renyi_prop}
    Let $q \ge 1$ and let $\mu$, $\nu$ be probability measures.
    \begin{enumerate}
        \item (Positivity) $\Ren_q(\mu \mmid \nu) \ge 0$, with equality if and only if $\mu = \nu$.
        \item (Monotonicity) R\'enyi divergences are increasing in the order, i.e., $q \mapsto \Ren_q(\mu \mmid \nu)$ is increasing.
        \item (Data processing inequality) For any Markov kernel $P$, it holds that $\Ren_q(\mu P \mmid \nu P) \le \Ren_q(\mu \mmid \nu)$.
        \item (Convexity) The divergences $\KL$ and $\Hell_q$ (for $q > 1$) are jointly convex.
        \item (Gaussian identity) $\Ren_q(\cN(x,\Sigma) \mmid \cN(y,\Sigma)) = \frac{q}{2}\,\langle x-y, \Sigma^{-1}\,(x-y)\rangle$.
    \end{enumerate}
\end{theorem}

\subsection{Shifted composition rule}\label{ssec:scr}

Next, we recall the \emph{shifted composition rule} from~\cite{scr1}, which will form the basis of our arguments in \S\ref{sec:disc}. Write $\Coup(\mu,\nu)$ for the set of couplings of two probability measures $\mu \in \mc P(\Omega_1)$, $\nu \in \mc P(\Omega_2)$, i.e., the set of probability measures $\gamma \in \mc P(\Omega_1\times \Omega_2)$ whose marginals are $\mu$ and $\nu$ respectively.

\begin{theorem}[Shifted composition rule]\label{thm:scr}
    Let $X$, $X'$, $Y$ be three jointly defined random variables on a standard probability space $\Omega$.
    Let $\bs\mu$, $\bs \nu$ be two probability measures over $\Omega$, with superscripts denoting the laws of random variables under these measures.
    \begin{enumerate}
        \item (Shifted chain rule) It holds that
        \begin{align*}
            \KL(\bs\mu^Y \mmid \bs \nu^Y)
            &\le \KL(\bs\mu^{X'} \mmid \bs \nu^X) + \inf_{\gamma \in \Coup(\bs \mu^X, \bs \mu^{X'})} \int \KL(\bs\mu^{Y\mid X=x} \mmid \bs \nu^{Y\mid X=x'}) \,\gamma(\D x, \D x')\,.
        \end{align*}
        \item Let $q \ge 1$. Then, it holds that
        \begin{align*}
            \Ren_q(\bs\mu^Y \mmid \bs \nu^Y)
            &\le \Ren_q(\bs\mu^{X'} \mmid \bs \nu^X) + \inf_{\gamma \in \Coup(\bs \mu^X, \bs \mu^{X'})} {\esssup{\gamma}_{(x,x') \in \Omega \times \Omega}{\Ren_q(\bs\mu^{Y\mid X=x} \mmid \bs \nu^{Y\mid X=x'})}}\,.
        \end{align*}
    \end{enumerate}
\end{theorem}

We briefly describe the intuition behind this principle.
We think of $\bs\mu$ and $\bs\nu$ as describing two possible laws of the stochastic process $X\to Y$, and we wish to bound the divergence between the laws of $Y$ under $\bs\mu$ and $\bs\nu$ respectively.
One natural way to do so is to apply the chain rule (for $\KL$) or the composition rule (for $\Ren_q$, $q > 1$), but this does not always lead to a bound that is tight or even non-vacuous.
Instead, the shifted composition rule allows for the introduction of an auxiliary ``shift'', which replaces $X$ with $X'$ under $\bs\mu$ (c.f.\ the first term in the above inequalities), at a price encapsulated by how distinguishable $X$ is from $X'$ (c.f.\ the second term).
As we demonstrate in \S\ref{sec:disc}, the freedom afforded by this auxiliary shift allows for capturing coupling arguments with information-theoretic tools.

\subsection{Convexity principle}\label{ssec:prelim:convexity}

In this paper, we establish inequalities of the form
\begin{align}\label{eq:cvxty_principle_diracs}
    \Ren_q(\delta_x P * \delta_v \mmid \delta_x P) \le \rho(v)\,, \qquad\text{for all}~v\in\R^d\,,
\end{align}
where $P$ is a Markov kernel on $\R^d$.
In this section, we demonstrate that via the joint convexity of $f$-divergences, such inequalities immediately imply more general ones where the Dirac measures are replaced by arbitrary measures.

\begin{lemma}[Convexity principle]\label{lem:cvxty}
    Suppose that an inequality of the form~\eqref{eq:cvxty_principle_diracs} holds.
    Then, for any probability measures $\mu$, $\nu$, $\nu'$ on $\R^d$, the following hold.
    \begin{enumerate}
        \item If $q = 1$ (i.e., $\Ren_q = \KL$), then
        \begin{align*}
            \KL(\mu P * \nu \mmid \mu P * \nu')
            &\le \inf_{\gamma \in \Coup(\nu,\nu')}\int \rho(v-v')\,\gamma(\D v,\D v')\,.
        \end{align*}
        \item If $q > 1$, then
        \begin{align*}
            \Ren_q(\mu P * \nu \mmid \mu P * \nu')
            &\le \inf_{\gamma\in\Coup(\nu,\nu')} \frac{1}{q-1} \log \int \exp\bigl((q-1)\,\rho(v-v')\bigr)\,\gamma(\D v, \D v')\,.
        \end{align*}
    \end{enumerate}
\end{lemma}
\begin{proof}
    First, note that since $x\mapsto x - v'$ is a diffeomorphism, the data-processing inequality (see Theorem~\ref{thm:renyi_prop}) and~\eqref{eq:cvxty_principle_diracs} imply that
    \begin{align*}
        \KL(\delta_x P * \delta_v \mmid \delta_x P * \delta_{v'})
        &= \KL(\delta_x P * \delta_{v-v'} \mmid \delta_x P)
        \le \rho(v-v')\,, \qquad\text{for all}~v,v'\in\R^d\,.
    \end{align*}
    Let $\gamma$ be a coupling of $\nu$ and $\nu'$.
    By the joint convexity of the KL divergence (Theorem~\ref{thm:renyi_prop}),
    \begin{align*}
        \KL(\mu P * \nu \mmid \mu P * \nu')
        &= \KL\Bigl(\int \delta_x P * \delta_v \,\mu(\D x)\,\gamma(\D v, \D v') \Bigm\Vert \int \delta_x P * \delta_{v'} \, \mu(\D x)\,\gamma(\D v, \D v')\Bigr) \\
        &\le \int \KL(\delta_x P * \delta_v \mmid \delta_x P * \delta_{v'}) \,\mu(\D x)\,\gamma(\D v, \D v')
        \le \int \rho(v-v')\,\gamma(\D v, \D v')\,.
    \end{align*}
    The proof for $\Ren_q$, $q > 1$ is similar, except that we use the joint convexity of $\Hell_q$ defined in~\eqref{eq:Hell_q} (note that $\Ren_q$ itself is not jointly convex).
\end{proof}

\begin{remark}[Refinements]\label{rem:convexity-refinements}
    Although we do not dwell on this point here, we briefly remark that several refinements are possible.
    In our prior work~\cite{scr1}, we showed that by taking advantage of the convexity of ${(\Hell_q + 1)}^{1/q}$ w.r.t.\ its first argument or the convexity of $\Ren_q$ w.r.t.\ its second argument, one can refine the convexity principle for the case $q > 1$, see \S 3.3.
    The same can be applied here.

    In a complementary direction, one can incorporate dependence between the initial condition $x$ and the terminal shift $v$, by generalizing from measures of the form $\mu P \ast \nu$ to measures of the form $\int P(x,\cdot) * Q(x, \cdot) \,\mu(\D x) = \int \delta_x P * \delta_v\, \mu(\D x)\,Q(x,\D v)$ for some Markov kernel $Q$ on $\R^d$. Indeed, let $\gamma_x$ be a coupling of two kernels $Q(x,\cdot)$ and $Q'(x,\cdot)$. Then, by the joint convexity of the KL divergence,
    \begin{align*}
        &\KL\Bigl(\int \delta_x P * \delta_v\,\mu(\D x)\,Q(x,\D v) \Bigm\Vert \int \delta_x P * \delta_{v'} \,\mu(\D x)\,Q'(x, \D v') \Bigr) \\
        &\qquad = \KL\Bigl(\int \delta_x P *\delta_v\, \mu(\D x)\,\gamma_x(\D v, \D v') \Bigm\Vert \int \delta_x P * \delta_{v'} \, \mu(\D x)\,\gamma_x(\D v, \D v')\Bigr) \\
        &\qquad \le \int \KL(\delta_x P * \delta_v \mmid \delta_x P * \delta_{v'}) \,\mu(\D x)\,\gamma_x(\D v, \D v')
        \le \int \rho(v-v') \,\mu(\D x)\,\gamma_x(\D v, \D v')\,.
    \end{align*}
    This observation can also be combined with the refinements mentioned in the preceding paragraph.
    However, for clarity of exposition, we focus on the simple version of the convexity principle as stated in Lemma~\ref{lem:cvxty}.   
\end{remark}

%% file: sections/discrete.tex

\section{Discrete-time arguments}\label{sec:disc}

\subsection{General principle}\label{ssec:disc:general}

In analog to the one-step-to-multi-step bound for backward regularity in~\cite[\S3.2]{scr1}, we now state a one-step-to-multi-step reduction for forward regularity. We remark that for forward regularity, the Markov kernel $P$ need only satisfy a $1$-step regularity bound---whereas for the backward regularity studied in~\cite{scr1}, it is necessary to additionally assume that $P$ is Wasserstein-Lipschitz. For simplicity, we consider here Dirac initializations and Dirac convolutions; this extends to general measures via the convexity principle in \S\ref{ssec:prelim:convexity}. 

\begin{theorem}[One-step-to-multi-step for forward regularity]\label{thm:one-to-multi}
Let $q\ge 1$. Suppose that $P$ is a Markov kernel on $\R^d$ satisfying the following $1$-step regularity bound for Dirac initializations:
\begin{align}
    \Ren_q( \delta_x P \ast \delta_v \mmid \delta_{x+w} P) \leq ( c_1\, \|w\| + c_2\,\|v-w\|)^2\,, \qquad \forall x,v,w \in \R^d\,.\label{eq-thm1:one-step}
\end{align}
Then for all $x,v \in \R^d$,
\begin{align}
    \Ren_q( \delta_x P^N \ast \delta_{v} \mmid \delta_x P^N) \leq \frac{1 - {(1 - c_1/c_2)}^2}{1 - {(1 - c_1/c_2)}^{2N}} \, c_2^2\,\|v\|^2\,.
    \label{eq-thm1:multi-step}
\end{align}
\end{theorem}

Although the condition~\eqref{eq-thm1:one-step} may seem mysterious at first sight, it is motivated by the one-step regularity bound which is satisfied by discretized It\^o diffusions---as it is trivial to verify in that setting, since it reduces to a simple divergence calculation between Gaussians.
Theorem~\ref{thm:one-to-multi} then immediately implies tight forward regularity bounds for It\^o diffusions (in both discrete and continuous settings). See \S\ref{ssec:disc:ito} for details.

We also remark that the condition~\eqref{eq-thm1:one-step} admits an interesting interpretation when $v = w$: it measures the extent to which convolution with $\delta_w$ commutes with the Markov kernel $P$.

\begin{proof}[Proof of Theorem~\ref{thm:one-to-multi}]
    Consider any sequence of shifts $v_0, \dots, v_N$ satisfying $v_0 = 0$ and $v_N = v$.
    For all $n \in \{0, \dots, N-1\}$, 
    \begin{align}
        \Ren_q( \delta_x P^{n+1} \ast \delta_{v_{n+1}} \mmid \delta_x P^{n+1})
        &\leq
        \Ren_q( \delta_x P^n \ast \delta_{v_n} \mmid \delta_x P^n)
        +
        \sup_{z\in\R^d}
        \Ren_q( \delta_{z} P \ast \delta_{v_{n+1}} \mmid \delta_{z + v_n} P) \nonumber
        \\
        &\leq
        \Ren_q( \delta_x P^n \ast \delta_{v_n} \mmid \delta_x P^n)
        +
        ( c_1\, \|v_n\| + c_2\, \|v_{n+1} - v_n\|)^2
        \,.\label{eq-pf-thm1:claim}
    \end{align}
    Above, the second inequality uses the one-step regularity~\eqref{eq-thm1:one-step}. The first inequality is by an application of the shifted composition rule (Theorem~\ref{thm:scr}) where $\bs{\mu}$ is the joint distribution under which $X \sim \delta_x P^n \ast \delta_{v_{n+1}}$, $X' \sim \delta_x P^n \ast \delta_{v_n}$, and $Y \sim \delta_x P^{n+1} \ast \delta_{v_{n+1}}$; and $\bs{\nu}$ is the joint distribution under which $X \sim \delta_x P^n$ and $Y \sim \delta_x P^{n+1}$. In this application of the shifted composition rule, we simplify the infimum over couplings $\gamma \in \Coup(\bs{\mu}^{X}, \bs{\mu}^{X'}) = \Coup(\delta_x P^n \ast \delta_{v_{n+1}}, \delta_x P^n \ast \delta_{v_n})$ by taking the deterministic coupling $\gamma$ which pairs $(z + v_{n+1}, z + v_n)$. 
    \par Iterating the inequality~\eqref{eq-pf-thm1:claim} yields
    \begin{align*}
        \Ren_q( \delta_x P^N \ast \delta_{v_N} \mmid \delta_x P^N)
        \leq
        \Ren_q( \delta_x \ast \delta_{v_0} \mmid \delta_x)
        + 
        \sum_{n=0}^{N-1} ( c_1\, \|v_n\| + c_2\, \|v_{n+1} - v_n\| )^2\,.
    \end{align*}
    By simplifying the R\'enyi terms using the boundary conditions $v_0 = 0$ and $v_N = v$, and optimizing over the intermediate shifts $v_1, \dots, v_{N-1}$, we conclude
    \begin{align}
        \Ren_q( \delta_x P^N \ast \delta_{v} \mmid \delta_x P^N)
        \leq
        \inf_{\substack{v_0, \dots, v_N \\ \text{s.t. }v_0 = 0, \, v_N = v} } \sum_{n=0}^{N-1} ( c_1\, \|v_n\| + c_2\, \|v_{n+1} - v_n\|)^2\,.
    \end{align}
    Reparameterizing to a sequence $a_0 \leq \cdots \leq a_n$ via $v_n = a_n v$ yields
    \begin{align}
        \Ren_q( \delta_x P^N \ast \delta_{v} \mmid \delta_x P^N)
        \leq
        \|v\|^2 \, \inf_{\substack{a_0 \leq \cdots \leq a_N  \\ \text{s.t. }a_0 = 0, \, a_N = 1} } \sum_{n=0}^{N-1} \big( c_1\, a_n + c_2\, (a_{n+1} - a_n) \big)^2\,.
    \end{align}
    Further reparameterizing to $b_{n+1} = a_{n+1} - ra_{n}$ where $r \deq 1 - c_1/c_2$ yields
    \begin{align}
        \Ren_q( \delta_x P^N \ast \delta_{v} \mmid \delta_x P^N)
        \leq
        c_2^2\, \|v\|^2 \, \inf_{\substack{b_1, \dots, b_N  \\ \text{s.t. } \sum_{n=1}^N r^{N-n}b_n = 1 \\ b_{n+1} \geq (1-r) \sum_{i=1}^n b_i r^{n-i}, \;\; \forall n < N }} \sum_{n=1}^N b_n^2 \,.
    \end{align}
    This optimization is easily solved in closed form (in the control theory literature this falls under the class of ``LQR'' problems). Even without the inequality constraints, the unique optimal solution is $b_i = r^{N-i}\, (1 - r^2)/(1 - r^{2N})$, yielding the optimal value $
    (1 - r^2)/(1 - r^{2N})$. Plugging this in completes the proof.
\end{proof}

\subsection{Forward regularity of It\^o diffusions}\label{ssec:disc:ito}

Here we illustrate how the techniques developed in \S\ref{ssec:disc:general} immediately yield tight forward regularity bounds for It\^o diffusions and their discretizations. We show tightness in \S\ref{app:tightness}. This is the first tight regularity bound for (continuous) It\^o diffusions, and to our knowledge, it is moreover the first forward regularity bound in the literature for discretized It\^o diffusions.

\par In what follows, consider the It\^o SDE setting described at the end of \S\ref{sec:intro}, let $(P_t)_{t \geq 0}$ denote the associated semigroup, and let $\hat{P}_h$ denote the Markov kernel corresponding to the discretized diffusion with time step $h > 0$, i.e., 
\begin{align}
    \hat{P}_h(x,\cdot) = \cN(x + h\,b(x), \,h\, \sigma \sigma^\T)\,. 
\end{align}

\begin{theorem}[Forward regularity for discretized It\^o diffusions]\label{thm:reg-disc} 
Denote $r \deq 1 - Lh$. For any step size $h < 1/L$, any number of steps $N \in \N$, any initialization measure $\mu$, and any convolution measures $\nu,\nu'$, it holds that
    \begin{align*}
        \KL( \mu \hat{P}_h^N \ast \nu \mmid \mu \hat{P}_h^N \ast \nu')
        \leq
        \frac{1 - r^2}{2\lambda h\,(1 - r^{2N})}\, W_2^2(\nu,\nu')\,,
    \end{align*}
    and for any $q \ge 1$,
    \begin{align}
        \Ren_q( \mu \hat{P}_h^N \ast \nu \mmid \mu \hat{P}_h^N \ast \nu')
        \leq \inf_{\gamma \in \Coup(\nu,\nu')} \frac{1}{q-1} \log \int \exp\Bigl( \frac{q\,(q-1)\,(1 - r^2)}{2\lambda h \,(1 - r^{2N})}\,\norm{v-v'}^2\Bigr)\, \gamma(\D v, \D v') \,.
    \end{align}
\end{theorem}
\begin{proof}
    We show the claim for R\'enyi divergence; the claim for KL divergence follows by an identical argument or by taking the limit $q \searrow 1$. 
    \par First, note that $\hat{P}_h$ satisfies the one-step regularity~\eqref{eq-thm1:one-step} with $c_1 = \tfrac{L \sqrt{qh}}{\sqrt{2\lambda}}$ and $c_2 = \tfrac{\sqrt{q}}{\sqrt{2h\lambda}}$ since
    \begin{align*}
        \Ren_q( \delta_x \hat{P}_h \ast \delta_v \mmid \delta_{x+w} \hat{P}_h) 
        &=
        \Ren_q \bigl( \cN(x + h\,b(x) + v, h\,\sigma \sigma^\T) \bigm\Vert \cN(x + w + h\,b(x+w), h\,\sigma \sigma^\T) \bigr) 
        \\ &\le \frac{q}{2\lambda h}\, \|h\,b(x) - h\,b(x+w) + v-w\|^2
        \\ &\leq \frac{q}{2\lambda h}\, (Lh\, \|w\| + \|v-w\|)^2\,.
    \end{align*}
    Above, the first step is by definition of $\hat{P}_h$. The second step is by the identity for the R\'enyi divergence between Gaussians (Theorem~\ref{thm:renyi_prop}) and the uniform ellipticity $\sigma \sigma^\T \succeq \lambda I$. The third step is by using the triangle inequality and $L$-Lipschitzness of the drift $b$. 
    \par By Theorem~\ref{thm:one-to-multi}, this one-step regularity implies the multi-step regularity
    \begin{align}
        \Ren_q( \delta_x \hat{P}_h^N \ast \delta_{v} \mmid \delta_x \hat{P}_h^N) \leq \frac{1 - {(1 - c_1/c_2)}^2}{1 - {(1 - c_1/c_2)}^{2N}}\, c_2^2\, \|v\|^2 
        = \frac{q\,(1 - r^2)}{2\lambda h\,(1 - r^{2N})}\, \|v\|^2\,.\label{eq:disc-dirac}
    \end{align}
    The statements of the theorem now follow from the convexity principle (Lemma~\ref{lem:cvxty}).
\end{proof}

This tight forward regularity result for the discretized semigroup immediately implies the following tight regularity result for the (standard, continuous-time) semigroup by taking the limit as the step size $h \searrow 0$ and the total elapsed continuous time is fixed to $T = Nh$. 

\begin{cor}[Forward regularity for It\^o diffusions]\label{cor:reg-cont}
    For any time $T$, any initialization measure $\mu$, and any convolution measures $\nu,\nu'$,
     \begin{align}
        \KL( \mu P_T \ast \nu \mmid \mu P_T \ast \nu')
        \leq
        \frac{L}{\lambda \,(1 - \exp(-2L T))}\, W_2^2(\nu,\nu')\,,
    \end{align}
    and for any $q \ge 1$,
    \begin{align}
        \Ren_q( \mu P_T \ast \nu \mmid \mu P_T \ast \nu')
        \leq
        \inf_{\gamma \in \Coup(\nu,\nu')} \frac{1}{q-1} \log \int \exp\Bigl( \frac{L q\,(q-1)\, \|v-v'\|^2}{\lambda\,(1 - \exp(-2L T))} \Bigr)\, \gamma(\D v, \D v')\,.
    \end{align}
\end{cor}

\begin{remark}[Refined convexity principle]
    The results in this section are proved for arbitrary initialization measures $\mu$ and arbitrary convolution measures $\nu,\nu'$ by using the convexity principle (Lemma~\ref{lem:cvxty}) to ``upgrade'' the forward regularity bound~\eqref{eq:disc-dirac} from the setting where $\mu,\nu,\nu'$ are all Dirac. As discussed in Remark~\ref{rem:convexity-refinements}, it is straightforward to instead use the refined version of the convexity principle from~\cite[\S3.3]{scr1} to obtain tighter bounds for specific settings of $\mu,\nu,\nu'$; and/or use a coupled version of the convexity principle to incorporate dependencies between $\mu,\nu,\nu'$.
\end{remark}

\subsection{Dual shifted divergences}\label{ssec:discrete:dual-sd}

Shifted divergences---an information-theoretic notion originating from the differential privacy literature~\cite{pabi}---are intimately connected to backward regularity, as pointed out in Part I~\cite[\S3.5]{scr1}. Here we briefly recall this connection in order to point out a parallel connection between between forward regularity and a new ``dual'' version of shifted divergences.

\paragraph*{Shifted divergences and backward regularity.} We begin by recalling the definition of the shifted divergence between distributions $\mu$ and $\nu$:
\begin{align}
        \mathcal{D}^{(z)}(\mu \mmid \nu)  \deq  \inf_{\mu' \in \mathcal{P}(\R^d) \; : \; W_{\infty}(\mu,\mu') \leq z} \mathcal{D}(\mu' \mmid \nu)\,.\label{eq:sd-def}
\end{align}
In words, the shift $z \geq 0$ enables changing one argument of the divergence $\cD$ (typically KL or R\'enyi) in Wasserstein distance. Shifted divergences are intimately related to backward regularity since it can be shown that
\begin{align}
    \mathcal{D}(\delta_x P \mmid \delta_y P) \leq \mathcal{D}^{(z)}(\delta_x \mmid \delta_y) + c_P z^2\,,
    \label{eq:sd-ineq1}
\end{align}
for Markov semigroups $P$ such as the Langevin kernel, in both discrete-time~\cite{AltTal23Langevin} and continuous-time settings~\cite{scr1}. Indeed, choosing $z = \|x-y\|^2$ yields the inequality
\begin{align}
    \mathcal{D}(\delta_x P \mmid \delta_y P) \leq c_P\, \|x-y\|^2\,,
    \label{eq:sd-ineq2}
\end{align}
which is one of the many equivalent statements of backward regularity for $P$---the other equivalences including reverse transport inequalities (equivalent by the convexity principle), Harnack inequalities (equivalent by H\"older duality), curvature-dimension inequalities (equivalent by semigroup methods), etc. We refer the reader to Part I~\cite{scr1} for further discussion.

\paragraph*{Dual shifted divergences and forward regularity.} In analogy, we now define a ``dual'' notion of shifted divergences, with negative shift $-z \leq 0$:
\begin{align}
    \mathcal{D}^{(-z)}(\mu \mmid \nu)
         \deq 
        \sup_{\rho \in \mathcal{P}(\R^d) \; : \; W_\infty(\rho,\delta_0) \leq z} \mathcal{D}(\mu \ast \rho \mmid \nu)
        =
        \sup_{v \in \R^d \; : \; \|v\| \leq z} \mathcal{D}(\mu \ast \delta_{v} \mmid \nu)
        \,.
    \label{eq:dualsd-def}
\end{align}
We pause to give a few remarks about this definition. First, now that the shift is negative, the infimum becomes a supremum. Second,~\eqref{eq:dualsd-def} gives two expressions for this definition---these are equivalent by quasi-convexity of the R\'enyi divergence and the fact that a convolution with an arbitrary distribution can be equivalently viewed as a mixture over convolutions with Diracs. 
Third, rather than ranging $\mu'$ over a Wasserstein ball around $\mu$, here $\mu'$ is further restricted to be of the form $\mu \ast \rho$ where $\rho$ is supported on a small ball. This is necessary for the dual shifted divergence to be meaningful (e.g., else one can find $\mu' \not \ll \nu$, rendering the dual shifted divergence infinite). 

\par The key point of this discussion is that, similarly to the interpretation~\eqref{eq:sd-ineq2} of backward regularity in terms of the standard shifted divergence, the forward regularity we proved in \S\ref{ssec:disc:ito} (see~\eqref{eq:disc-dirac} or Theorem~\ref{thm:reg-disc}) can be interpreted as follows in terms of the dual shifted divergence:
\begin{align}
    \mathcal{D}^{(-z)}(\delta_x P \mmid \delta_x P) \leq c_P z^2\,.
    \label{eq:sd-dual-ineq}
\end{align}

\paragraph*{A unified family of shifted divergences.} Our notation for dual shifted divergences intentionally unifies this notion with the standard version of shifted divergences. Now the standard shifted divergences are extended by taking an arbitrary $z \in \R$, not just $z \geq 0$. This unification applies not just to the definition, but also to some key properties. In particular, the analysis of shifted divergences relies on two lemmas: the convolution lemma~\cite[Lemma 20]{pabi} and the contraction lemma~\cite[Lemma 21]{pabi}. Although the latter does not appear to extend to dual shifted divergences---necessitating our analysis for forward regularity in~\S\ref{ssec:disc:ito} based on the shifted composition rule---the former does. 
\par We state this lemma below in a unified way that applies to both the original and dual versions of shifted divergences. Below, let $S_q(\xi,a)  \deq  \sup_{v\,:\,\|v\| \leq a} \Ren_q(\xi \ast \delta_v \mmid \xi)$ denote the R\'enyi sensitivity of a distribution $\xi$ with respect to a translation of size at most $a$. The proof is given in \S\ref{app:dual-sd}.

\begin{lemma}[Generalized convolution lemma for shifted divergences]\label{lem:sd-dual:convolution}
    For any initial shift $z \in \R$, shift increase $a \geq 0$, R\'enyi parameter $q \geq 1$, and distributions $\mu,\nu,\xi$,
    \[
        \Ren_q^{(z)}\left( \mu \ast \xi \mmid \nu \ast \xi \right)
        \leq
         \Ren_q^{(z+a)}(\mu \mmid \nu) + 
          S_q(\xi, a)\,.      
    \]
\end{lemma}

%% file: sections/continuous.tex

\section{Continuous-time arguments}\label{sec:cont}

In this section, we give the corresponding continuous-time coupling arguments for establishing forward regularity.
Since it does not particularly complicate the notation nor the proof, we work with a slightly more general setting than the one described at the end of \S\ref{sec:intro}, i.e., we consider It\^o SDEs with time-varying coefficients
\begin{align}\label{eq:ito_diffusion}
    \D X_t
    &= b_t(X_t) \, \D t + \sigma_t\,\D B_t\,, \qquad X_0 = x\,.
\end{align}
We assume that $b_t$ is $L$-Lipschitz and that $\sigma_t \sigma_t^\T \succeq \lambda I$ for all $t\ge 0$, and for simplicity we take $\sigma_t \in \R^{d\times d}$ to be an invertible matrix for each $t\ge 0$. We also assume, of course, that~\eqref{eq:ito_diffusion} admits a unique strong solution, which is guaranteed if the coefficients are also Lipschitz in time.
Throughout, we let $\mu_t \deq \law(X_t)$ denote the marginal law of the process.

We define the auxiliary process $Y_t \deq X_t + a_t v$ for $t \in [0,T]$, where $\{a_t\}_{t\in [0,T]}$ is increasing and chosen so that $a_0 = 0$ and $a_T = 1$.
This ensures that $Y_0 = x$ and $Y_T = X_T + v$.
By It\^o's formula,
\begin{align}\label{eq:cont_time_aux}
    \D Y_t
    &= \D X_t + \dot a_t v \, \D t
    = (b_t(Y_t - a_t v) + \dot a_t v) \, \D t + \sigma_t \, \D B_t\,.
\end{align}
This auxiliary process $\{Y_t\}_{t\in [0,T]}$ is the analogue, in the continuous-time limit $h \searrow 0$, of the auxiliary process used in the proof of Theorem~\ref{thm:one-to-multi}.

\par Our goal is to bound the R\'enyi divergence $\Ren_q(\mu_T * \delta_v \mmid \mu_T) = \Ren_q(\law(Y_T) \mmid \law(X_T))$. In discrete time, this was achieved using the shifted composition rule. Here we provide two continuous-time analogues. One is based on Girsanov's theorem and stochastic calculus arguments (\S\ref{ssec:cont:synchronous}), and the other is based on direct differentiation of the quantity $\Ren_q(\law(Y_t) \mmid \law(X_t))$ with respect to $t$ (\S\ref{ssec:cont:wasserstein}).

\begin{remark}
    For backward regularity,~\cite[\S4]{scr1} provides two continuous-time proofs, based on two different constructions of the auxiliary process: the ``synchronous coupling'' and ``Wasserstein coupling''. For forward regularity, the auxiliary process is simply a deterministic shift of the original process, so the coupling is trivial and these two couplings coincide. Nevertherless, the two couplings in~\cite{scr1} lead to two distinct proof techniques which are reflected in the following subsections.
\end{remark}

\subsection{Proof via Girsanov's theorem}\label{ssec:cont:synchronous}

In the first approach, we control the quantity $\Ren_q(\law(Y_T) \mmid \law(X_T))$ via Girsanov's theorem~\cite[Theorem 5.22]{legall2016stochasticcalc}, since the process $\{Y_t\}_{t\in [0,T]}$ can be realized as a Girsanov transformation of $\{X_t\}_{t\in [0,T]}$.
This approach was introduced by F.-Y.\ Wang in~\cite{Wang14ShiftHarnack} and has been successfully applied to a multitude of settings, see \S\ref{sec:intro} for references.
Crucially, we show how to apply his method to obtain sharp constants, which were not previously known.

\paragraph*{KL divergence bound.}
We begin with the KL divergence ($q=1$), since the argument is simpler in this case.
By Girsanov's theorem, if $\bs\mu_T'$ and $\bs\mu_T$ denote the path measures corresponding to $\{Y_t\}_{t\in [0,T]}$ and $\{X_t\}_{t\in [0,T]}$ respectively,
\begin{align}
    \KL(\mu_T * \delta_v \mmid \mu_T)
    &\le \KL(\bs\mu_T' \mmid \bs\mu_T)
    = \frac{1}{2}\, \E\int_0^T \norm{\sigma_t^{-1}\,(b_t(X_t)  - b_t(X_t + a_t v) + \dot a_t v)}^2 \, \D t\,.
    \label{eq:fwd-reg:cont:sd}
\end{align}
Using the uniform ellipticity and the Lipschitzness of the drift, this is bounded by
\begin{align*}
    \frac{\norm v^2}{2\lambda} \int_0^T {(La_t + \dot a_t)}^2\,\D t\,.
\end{align*}
To minimize the integrand, we use calculus of variations.
Consider the functional
\begin{align*}
    \ms F(a)
    &\deq \int_0^T {(La_t + \dot a_t)}^2 \, \D t\,.
\end{align*}
The first variation of $\ms F$ is computed to be
\begin{align*}
    \delta \ms F(a)(t)
    &= 2\,(L^2 a_t - \ddot a_t)\,.
\end{align*}
We set the first variation to zero.
With the boundary conditions $a_0 = 0$, $a_T = 1$, this is solved with
\begin{align}\label{eq:cont_opt_shift}
    a_t
    &= \frac{\exp(Lt) - \exp(-Lt)}{\exp(LT) - \exp(-LT)} = \frac{\sinh(Lt)}{\sinh(LT)}\,.
\end{align}
Substituting this in, we obtain the bound
\begin{align*}
    \KL(\mu_T * \delta_v \mmid \mu_T)
    &\le \frac{2L^2\,\norm v^2}{\lambda\,{(\exp(LT) - \exp(-LT))}^2} \int_0^T \exp(2Lt) \, \D t
    = \frac{L\,\norm v^2}{\lambda\,(1- \exp(-2LT))}\,.
\end{align*}

\paragraph*{R\'enyi divergence bound.}
We now consider the R\'enyi divergence of order $q > 1$.
We define the $\bs\mu_T$-martingale $M$ via $M_t \deq \int \langle \sigma_t^{-1}\,(b(X_t - a_t v) - b_t(X_t) + \dot a_t v), \D B_t\rangle$, where $\{B_t\}_{t\in [0,T]}$ is a standard Brownian motion under $\bs\mu_T$.
By Girsanov's theorem,
\begin{align*}
    \Ren_q(\mu_T * \delta_v \mmid \mu_T)
    &\le \Ren_q(\bs \mu_T' \mmid \bs \mu_T)
    = \frac{1}{q-1} \log \E_{\bs \mu_T} \exp\bigl(qM_T - \frac{q}{2} \,{[M,M]}_T\bigr)\,.
\end{align*}
Recall that by It\^o's formula, for any martingale $\widetilde M$, the exponential martingale $\exp(\widetilde M - \frac{1}{2}\,[\widetilde M, \widetilde M])$ is a non-negative local martingale and hence a supermartingale.
Applying this to $\widetilde M \deq qM$, we have
\begin{align*}
    \E_{\bs \mu_T} \exp\bigl(qM_T - \frac{q}{2} \,{[M,M]}_T\bigr)
    &\le \bigl\lVert \exp\bigl(\frac{q\,(q-1)}{2}\,{[M, M]}_T\bigr)\bigr\rVert_{L^\infty(\bs\mu_T)}\,\underbrace{\E_{\bs\mu_T} \exp\bigl(qM_T - \frac{q^2}{2}\,{[M,M]}_T\bigr)}_{\le 1}\,.
\end{align*}
On the other hand,
\begin{align*}
    \bigl\lVert {[M, M]}_T\bigr\rVert_{L^\infty(\bs\mu_T)}  = \Bigl\lVert \int_0^T \norm{\sigma_t^{-1}\,(b(X_t - a_t v) - b_t(X_t) + \dot a_t v)}^2 \, \D t)\Bigr\rVert_{L^\infty(\bs\mu_T)} 
     \le \frac{\norm v^2}{\lambda}\int_0^T {(La_t + \dot a_t)}^2 \, \D t\,.
\end{align*}
Substituting in the same choice of $\{a_t\}_{t\in [0,T]}$ as in~\eqref{eq:cont_opt_shift}, and simplifying, we obtain
\begin{align*}
    \Ren_q(\mu_T * \delta_v \mmid \mu_T)
    &\le \frac{qL\,\norm v^2}{\lambda\,(1-\exp(-2LT))}\,.
\end{align*}


\subsection{Proof via divergence differentiation}\label{ssec:cont:wasserstein}

We now provide an argument based on differentiation along the Fokker{--}Planck equation which, to the best of our knowledge, is new.
In the following proof, the case of the KL divergence ($q=1$) is not substantially simpler than the case of general R\'enyi divergences ($q > 1$), so we provide the proof of the latter to avoid repetition.

The proof is based on the following observation.
Denoting $\mu_t' \deq \law(Y_t) \deq \law(X_t + a_t v)$, one can easily check that
\begin{align}\label{eq:fokker_planck}
    \begin{aligned}
        \partial_t \mu_t
        &= \frac{1}{2}\,\langle \sigma \sigma^\T, \nabla^2 \mu_t\rangle-\divergence(\mu_t\, b_t)\,, \\[0.25em]
        \partial_t \mu_t'
        &=\frac{1}{2}\,\langle \sigma\sigma^\T, \nabla^2 \mu_t'\rangle -\divergence(\mu_t' \,(b_t \circ \tau_{-a_t v} + \dot a_t v))\,,
    \end{aligned}
\end{align}
where for $u\in\R^d$, $\tau_u : \R^d\to\R^d$ denotes the translation $x\mapsto x + u$.
Indeed, $\{X_t\}_{t\in [0,T]}$ solves the SDE~\eqref{eq:ito_diffusion} and $\{Y_t\}_{t\in [0,T]}$ solves the SDE~\eqref{eq:cont_time_aux}, so the marginal laws of these processes solve the respective Fokker{--}Planck equations.

With~\eqref{eq:fokker_planck} in hand, we are now in a position to apply a lemma on differentiation of divergences along diffusions which has been utilized in prior works such as~\cite{VempalaW19, chenetal2022proximalsampler}.
We state a version of this lemma which generalizes~\cite[Lemma A.5]{scr1}; its proof follows exactly along the same lines.

\begin{lemma}
    Let $\psi : \R_+\to\R_+$ be twice continuously differentiable on $(0,\infty)$.
    Consider the associated divergence
    \begin{align*}
        \msf D_\psi(\mu \mmid \nu)
        &\deq \int \psi\bigl( \frac{\D\mu}{\D\nu}\bigr)\,\D\nu\,.
    \end{align*}
    Suppose that ${(\mu_t)}_{t\ge 0}$, ${(\nu_t)}_{t\ge 0}$ are positive and smooth densities evolving according to the equations
    \begin{align*}
        \partial_t \mu_t
        &= \frac{1}{2}\,\langle \sigma_t \sigma_t^\T, \nabla^2 \mu_t\rangle + \divergence(\mu_t a_t)\,, \\[0.25em]
        \partial_t \nu_t
        &= \frac{1}{2}\,\langle \sigma_t \sigma_t^\T, \nabla^2 \nu_t\rangle + \divergence(\nu_t b_t)\,.
    \end{align*}
    Here, ${(a_t)}_{t\ge 0}$ and ${(b_t)}_{t\ge 0}$ are families of vector fields on $\R^d$, and for each $t \ge 0$, $\sigma_t \in \R^{d\times d}$ is an invertible matrix.
    Then, it holds that
    \begin{align*}
        \partial_t \msf D_\psi(\mu_t \mmid \nu_t)
        &= -\E_{\mu_t}\bigl\langle \nabla \bigl(\psi' \circ \frac{\mu_t}{\nu_t}\bigr),\; \frac{1}{2}\,\sigma_t \sigma_t^\T \,\nabla \log \frac{\mu_t}{\nu_t} + a_t - b_t\bigr\rangle\,.
    \end{align*}
\end{lemma}

We apply this with $\psi(\cdot) = {(\cdot)}^q - 1$, for which $\msf D_\psi = \msf D_q$.
The simultaneous flow lemma yields
\begin{align*}
    \partial_t \msf D_q(\mu_t' \mmid \mu_t)
    &= -q\,\E_{\mu_t'}\Bigl\langle \nabla \bigl[ \bigl(\frac{\mu_t'}{\mu_t}\bigr)^{q-1}],\, \frac{1}{2}\,\sigma_t \sigma_t^\T\, \nabla \log\frac{\mu_t'}{\mu_t} + b_t - b_t \circ \tau_{-a_t v} - \dot a_t v\Bigr\rangle\,.
\end{align*}
By the chain rule,
\begin{align*}
    \E_{\mu_t'}\Bigl\langle \nabla \bigl[ \bigl(\frac{\mu_t'}{\mu_t}\bigr)^{q-1}], \,\sigma_t \sigma_t^\T\, \nabla \log\frac{\mu_t'}{\mu_t}\Bigr\rangle
    &= (q-1)\, \E_{\mu_t'}\Bigl[\bigl(\frac{\mu_t'}{\mu_t}\bigr)^{q-1} \,\bigl\lVert \nabla \log \frac{\mu_t'}{\mu_t}\bigr\rVert^2_{\sigma_t \sigma_t^\T}\Bigr]\,,
\end{align*}
where, for a positive definite matrix $\Sigma \succ 0$, we write $\norm x_\Sigma^2 \deq \langle x, \Sigma\,x\rangle$.
Also, by Young's inequality,
\begin{align*}
    &-\E_{\mu_t'}\Bigl\langle \nabla \bigl[ \bigl(\frac{\mu_t'}{\mu_t}\bigr)^{q-1}], \,b_t - b_t \circ \tau_{-a_t v} - \dot a_t v\Bigr\rangle
    = -(q-1)\,\E_{\mu_t'}\Bigl[ \bigl( \frac{\mu_t'}{\mu_t}\bigr)^{q-1}\, \bigl\langle \nabla \log \frac{\mu_t'}{\mu_t}, \,b_t - b_t \circ \tau_{-a_t v} - \dot a_t v\bigr\rangle\Bigr] \\
    &\qquad\le (q-1)\,\Bigl\{\frac{1}{2}\,\E_{\mu_t'}\Bigl[\bigl(\frac{\mu_t'}{\mu_t}\bigr)^{q-1} \,\bigl\lVert \nabla \log \frac{\mu_t'}{\mu_t}\bigr\rVert^2_{\sigma_t \sigma_t^\T}\Bigr] +\frac{1}{2}\,\E_{\mu_t'}\Bigl[\bigl(\frac{\mu_t'}{\mu_t}\bigr)^{q-1}\,\norm{b_t - b_t \circ \tau_{-a_t v} - \dot a_t v}^2_{(\sigma_t \sigma_t^\T)^{-1}}\Bigr]\Bigr\}\,.
\end{align*}
Hence,
\begin{align*}
    \partial_t \msf D_q(\mu_t' \mmid \mu_t)
    &\le \frac{q\,(q-1)}{2}\,\E_{\mu_t'}\Bigl[\bigl(\frac{\mu_t'}{\mu_t}\bigr)^{q-1}\,\norm{b_t - b_t \circ \tau_{-a_t v} - \dot a_t v}^2_{(\sigma_t \sigma_t^\T)^{-1}}\Bigr] \\
    &\le \frac{q\,(q-1)\,\norm v^2}{2\lambda}\,{(La_t + \dot a_t)}^2\,\E_{\mu_t'}\bigl[\bigl(\frac{\mu_t'}{\mu_t}\bigr)^{q-1}\bigr]\,.
\end{align*}
Differentiating the R\'enyi divergence via the chain rule, noting that $\Ren_q = \frac{1}{q-1} \log(1+\msf D_q)$,
\begin{align*}
    \partial_t\Ren_q(\mu_t' \mmid \mu_t)
    &\le \frac{q\,\norm v^2}{2\lambda}\,{(La_t + \dot a_t)}^2
\end{align*}
and therefore
\begin{align*}
    \Ren_q(\mu_T * \delta_v \mmid \mu_T)
    &\le \frac{q\,\norm v^2}{2\lambda}\int_0^T {(La_t + \dot a_t)}^2 \, \D t\,.
\end{align*}
Choosing $\{a_t\}_{t\in [0,T]}$ as in~\eqref{eq:cont_opt_shift}, we again arrive at the sharp bound
\begin{align*}
    \Ren_q(\mu_T * \delta_v \mmid \mu_T)
    &\le \frac{qL\,\norm v^2}{\lambda\,(1-\exp(-2LT))}\,.
\end{align*}

%% file: sections/harnack.tex
\section{Shift Harnack inequalities and curvature upper bounds}\label{sec:harnack}

We begin with brief background on shift Harnack inequalities in \S\ref{sec:shift_harnack_bg} (see also the discussion in \S\ref{sec:intro}). We recall the duality between shift Harnack inequalities and reverse transport inequalities in \S\ref{ssec:prelim:duality} which, together with our results from \S\ref{sec:disc} and \S\ref{sec:cont}, lead to our shift Harnack inequalities for the Langevin diffusion with \emph{optimal} constants.
We leverage these sharp inequalities to prove our main Theorem~\ref{thm:equiv} detailing equivalences and implications between the shift Harnack inequalities, a local gradient-entropy bound, and curvature bounds. Finally, implications for the stationary distribution are discussed in \S\ref{sec:implications_stationary}.

\subsection{Background on shift Harnack inequalities}\label{sec:shift_harnack_bg}

We now apply our results to the study of \emph{shift Harnack inequalities}, which were introduced by F.-Y.\ Wang in~\cite{Wang14ShiftHarnack}.
In particular, for a given Markov semigroup ${(P_t)}_{t\ge 0}$, we consider inequalities which take the form
\begin{align}\label{eq:bg_shift_har}
    {P_t(f(\cdot + v))}^p
    &\le P_t(f^p)\exp\bigl(C_p(t)\,\norm v^2\bigr)\,,
\end{align}
for all $v\in\R^d$, all $t > 0$, and all non-negative functions $f : \R^d\to\R_{\geq 0}$.
We also consider the shift log-Harnack inequality of the form
\begin{align}\label{eq:bg_shift_log_har}
    P_t(f(\cdot + v))
    &\le C_{\log}(t)\,\norm v^2 + \log P_t(\exp f)\,,
\end{align}
again for all $v\in\R^d$, all $t > 0$, and all non-negative functions $f : \R^d\to\R_{\geq 0}$.

Recall from \S\ref{sec:intro} that whereas Harnack inequalities encode regularity for Kolmogorov's backward equation, shift Harnack inequalities were introduced to encode the regularity of Kolmogorov's forward equation, and therefore yield information about the Lebesgue density of the marginal law of the stochastic process.
For instance, the validity of either~\eqref{eq:bg_shift_har} or~\eqref{eq:bg_shift_log_har} implies that for each $t > 0$, the semigroup $P_t$ admits a transition density $p_t$ with respect to Lebesgue measure which satisfies certain bounds, see~\cite[\S 1.4.2]{Wang13HarnackSPDE} for further applications such as heat kernel bounds.
For intuition, we illustrate this principle with a simple example below.

However, whereas Harnack inequalities are known to be \emph{equivalent} to the curvature-dimension condition, as well as to Wasserstein contraction and other functional inequalities such as the local Poincar\'e/log-Sobolev inequalities and gradient commutation inequalities (c.f.~\cite{bakry2014analysis, Wang14Diffusion} or~\cite[\S 6.1]{scr1}), such equivalences have not been fleshed out for shift Harnack inequalities. We aim to explore this point further in \S\ref{sec:equiv}.

\paragraph*{An illustrative example.}
As we show in the subsequent section, shift (log) Harnack inequalities are dual to certain reverse transport inequalities.
In particular,~\eqref{eq:bg_shift_har} is equivalent to an inequality of the form
\begin{align}\label{eq:bg_rev_transport}
    \Ren_q(\delta_x P_t * \delta_v \mmid \delta_x P_t)
    &\le \bar C_q(t)\,\norm v^2
\end{align}
for another constant $\bar C_q(t) > 0$, where $q = p/(p-1)$ is the dual exponent.
Supposing now that~\eqref{eq:bg_rev_transport} holds for $p=q=2$, let $\phi : \R^d\to\R$ be a test function (i.e., smooth and compactly supported).
For any vector $v\in\R^d\setminus\{0\}$, we can estimate $\abs{\langle \E_{\delta_x P_t} \nabla \phi, v\rangle} = \abs{\langle P_t(\nabla \phi)(x), v\rangle}$ via
\begin{align*}
    \lim_{\varepsilon \searrow 0} \frac{\abs{\E_{\delta_x P_t}[ \phi(\cdot + \varepsilon v) - \phi(\cdot)]}}{\varepsilon}
    &= \lim_{\varepsilon \searrow 0} \frac{\abs{\E_{\delta_x P_t * \delta_{\varepsilon v}} \phi - \E_{\delta_x P_t}\phi}}{\varepsilon}
    \leq
    \|\phi\|_{L^2(\delta_x P_t)}\cdot \limsup_{\varepsilon \searrow 0} 
    \frac{\sqrt{\chi^2(\delta_x P_t *\delta_{\varepsilon v} \mmid \delta_x P_t)}}{\varepsilon}
    \\
    &\le \|\phi\|_{L^2(\delta_x P_t)} \cdot \limsup_{\varepsilon \searrow 0} \frac{\sqrt{\exp(\Ren_2(\delta_x P_t * \delta_{\varepsilon v} \mmid \delta_x P_t)) - 1}}{\varepsilon} \\
    &\le \sqrt{\bar C_2(t)}\, \norm\phi_{L^2(\delta_x P_t)}\,\norm v\,.
\end{align*}
This inequality establishes $\norm{\E_{\delta_x P_t}\nabla \phi} \le \sqrt{\bar C_2(t)}\,\norm \phi_{L^2(\delta_x P_t)}$ for all test functions $\phi$. Since we can formally write $\E_{\delta_x P_t}\nabla \phi = -\E_{\delta_x P_t}[\phi\,\nabla \log \delta_x P_t]$, this can be interpreted as an $L^2$ bound for the logarithmic gradient of the Lebesgue density of $\delta_x P_t$.
In \S\ref{sec:implications_stationary}, we make this viewpoint precise by deriving sharp high-probability bounds for the logarithmic gradient.

\subsection{Duality between shift Harnack inequalities and reverse transport inequalities}\label{ssec:prelim:duality}

Since we aim to deduce shift Harnack inequalities from the reverse transport inequalities established in preceding sections, we formally state and prove a duality principle below.

\begin{lemma}\label{lem:duality}
    Let $\mu$ be a probability measure over $\R^d$ and let $v\in\R^d$.
    \begin{enumerate}
        \item Let $C_{\log,v} > 0$. Then,
        \begin{align}\label{eq:general_shift_log_harnack}
            \mu(f(\cdot + v))
            &\le C_{\log,v} + \log \mu(\exp f) \qquad\text{for all}~f : \R^d\to\R
        \end{align}
        if and only if
        \begin{align*}
            \KL(\mu * \delta_v \mmid \mu)
            &\le C_{\log,v}\,.
        \end{align*}
        \item Let $q > 1$ and let $p \deq q/(q-1)$ denote its H\"older conjugate. Let $C_{p,v} > 0$.
        Then,
        \begin{align}\label{eq:general_shift_harnack}
            \mu(f(\cdot + v))
            &\le C_{p,v}\, {\mu(f^p)}^{1/p} \qquad\text{for all}~f : \R^d\to\R_{\geq 0}
        \end{align}
        if and only if
        \begin{align*}
            \Ren_q(\mu * \delta_v \mmid \mu)
            &\le \frac{q}{q-1} \log C_{p,v}\,.
        \end{align*}
    \end{enumerate}
\end{lemma}
\begin{proof}
    Define the linear functional $L_v : L^p(\mu) \to \R$ via $f \mapsto \mu(f(\cdot + v))$.
    Then, the best constant $C_{p,v}$ in the shift Harnack inequality~\eqref{eq:general_shift_harnack} equals the operator norm $\norm{L_v}_{L^p(\mu)\to\R}$.
    On the other hand, since $L_v f = \int f \, \D(\mu *\delta_v) = \int f \,\frac{\D(\mu * \delta_v)}{\D \mu}\,\D \mu$, by H\"older duality it follows that $C_{p,v} = \norm{\frac{\D(\mu*\delta_v)}{\D\mu}}_{L^q(\mu)}$, from which the second equivalence follows.

    Similarly, the first equivalence follows from the Donsker{--}Varadhan variational principle, which shows that the best constant $C_{\log,v}$ in the inequality~\eqref{eq:general_shift_log_harnack} equals $\KL(\mu * \delta_v \mmid \mu)$.
\end{proof}

\subsection{Relationship with local gradient-entropy bounds and curvature upper bounds}\label{sec:equiv}

By combining the sharp reverse transport inequalities from \S\ref{sec:disc} and \S\ref{sec:cont} with the duality principle from \S\ref{ssec:prelim:duality}, we are now able to obtain \emph{sharp} shift (log) Harnack inequalities.
Moreover, as discussed in \S\ref{sec:intro}, the sharpness of the inequalities allows us to explore their relationship with other functional inequalities (namely, a ``local gradient-entropy inequality'') and with curvature bounds.
For convenience, we restate and prove our main theorem to this effect.

\equivthm*

\begin{proof}
    $\text{\ref{CurvBDD}} \Rightarrow [\forall p > 1\;\text{\ref{SHp}}] \Rightarrow [\exists p > 1\;\text{\ref{SHp}}]$.
    The first implication immediately follows from the reverse transport inequalities established in~\S\ref{sec:disc} and \S\ref{sec:cont}, together with Lemma~\ref{lem:duality}.
    The second implication is trivial.

    $\text{\ref{SHp}} \Leftrightarrow \text{\ref{SRTq}}$ for $q = \frac{p}{p-1}$, and $\text{\ref{SHlog}} \Leftrightarrow \text{\ref{SRTlog}}$. This follows from the duality encapsulated in Lemma~\ref{lem:duality}.
    
    $[\forall p > 1\;\text{\ref{SHp}}] \Leftrightarrow$ \ref{LGE}.
    According to~\cite[Proposition 1.3.2]{Wang13HarnackSPDE}\footnote{Take $\delta_e = 0$ and $\beta_e(\delta, x) = \beta\,\norm e^2/\{2\delta\,(1-\exp(-2\beta t))\}$ therein.}, the validity of the shift Harnack inequalities for all $p > 1$ is equivalent to
    \begin{align*}
        \abs{\langle P_t \nabla f, v\rangle}
        &\le \eta\,\{P_t(f\log f) - P_t f \log P_t f\} + \frac{\beta\,\norm v^2}{2\eta\,(1-\exp(-2\beta t))}\,P_t f\,, \qquad\text{for all}~\eta > 0\,.
    \end{align*}
    Optimizing over the choice of $\eta$, this is equivalent to
    \begin{align*}
        \frac{\abs{\langle P_t \nabla f, v\rangle}}{\norm v}
        &\le \sqrt{\frac{2\beta}{1-\exp(-2\beta t)}\,P_t f\,\{P_t(f\log f) - P_t f \log P_t f\}}\,.
    \end{align*}
    Since this holds for every choice of $v\in\R^d\setminus \{0\}$, this is~\ref{LGE}.

    $[\exists p > 1\;\text{\ref{SHp}}] \Rightarrow \text{\ref{SHlog}}$. This implication is~\cite[Theorem 1.3.5]{Wang13HarnackSPDE}.

    $\text{\ref{SHlog}} \Rightarrow \text{\ref{CurvUB}}$. Finally, we show that the shift log-Harnack inequality implies back the curvature upper bound via Taylor expansion.
    
    Let $f : \R^d\to\R$ be a compactly supported and smooth function such that $\nabla f(x) = u \in \R^d$ and $\nabla^2 f(x) = 0$.
    We apply the shift log-Harnack inequality~\ref{SHlog} with $v = ctu$ for some $c > 0$ to be chosen later, yielding the inequality
    \begin{align}\label{eq:slharnack_implies_curv}
        P_t(f(\cdot + ctu))(x) \le \log P_t(\exp f)(x) + \frac{\beta c^2 t^2\, \norm{\nabla f(x)}^2}{2\,(1-\exp(-2\beta t))}\,.
    \end{align}
    We now perform a careful Taylor expansion.
    For the left-hand side,
    \begin{align*}
        \partial_t P_t(f(\cdot + ctu))(x)
        &= \ms LP_t(f(\cdot + ctu))(x) + c\,\langle P_t \nabla f(\cdot + ctu), u\rangle(x)\,,
    \end{align*}
    and hence, since $\nabla^2 f(x) = 0$, we obtain
    \begin{align*}
        \partial_t\big|_{t=0} P_t(f(\cdot + ctu))(x)
        &= \ms Lf(x) + c\,\norm{\nabla f(x)}^2\,, \\
        \partial_t^2\big|_{t=0} P_t(f(\cdot + ctu))(x)
        &= \ms L^2 f(x) + 2c\,\langle \nabla f(x), \ms L\nabla f(x)\rangle\,.
    \end{align*}
    For the first term on the right-hand side,
    \begin{align*}
        \partial_t \log P_t(\exp f)(x)
        &= \frac{\ms L P_t(\exp f)(x)}{\exp f(x)}\,,
    \end{align*}
    and a tedious calculation based on the diffusion chain rule $\ms L\phi(f) = \phi'(f)\,\ms Lf + \phi''(f)\,\norm{\nabla f}^2$, the carr\'e du champ identity $\ms L(fg) = f\,\ms L g + g\,\ms L f + 2\,\langle \nabla f,\nabla g\rangle$, and $\nabla^2 f(x) = 0$ yields
    \begin{align*}
        \partial_t\big|_{t=0} \log P_t(\exp f)(x)
        &= \frac{\ms L(\exp f)(x)}{\exp f(x)}
        = \ms L f(x) + \norm{\nabla f(x)}^2\,, \\
        \partial_t^2\big|_{t=0} \log P_t(\exp f)(x)
        &= \ms L^2 f(x) + \ms L(\norm{\nabla f}^2)(x) + 2\,\langle \nabla f(x), \nabla \ms L f(x)\rangle\,.
    \end{align*}
    We now set $c=2$, substitute these expansions into~\eqref{eq:slharnack_implies_curv}, and divide by $t^2$ to obtain
    \begin{align*}
        2\,\langle\nabla f(x), \ms L\nabla f(x)\rangle
        &\le \frac{1}{2}\,\ms L(\norm{\nabla f}^2)(x) + \langle \nabla f(x), \nabla \ms L f(x) \rangle \\
        &\qquad{} + \frac{1}{t}\,\Bigl(\frac{2\beta t}{1-\exp(-2\beta t)} -1 \Bigr)\,\norm{\nabla f(x)}^2 + o(1)\,.
    \end{align*}
    Sending $t\searrow 0$, it yields the inequality
    \begin{align*}
        2\,\langle\nabla f(x), \ms L\nabla f(x)\rangle
        &\le \frac{1}{2}\,\ms L(\norm{\nabla f}^2)(x) + \langle \nabla f(x), \nabla \ms L f(x) \rangle + \beta\,\norm{\nabla f(x)}^2\,.
    \end{align*}
    We now use the identities $\frac{1}{2}\,\ms L(\norm{\nabla f}^2) - \langle \nabla f, \nabla \ms L f\rangle = \Gamma_2(f,f)$, where $\Gamma_2$ denotes the iterated carr\'e du champ operator, and the commutation relation $\ms L \nabla f - \nabla \ms L f = \nabla^2 V \,\nabla f$.
    Hence,
    \begin{align*}
        2\,\langle \nabla f(x),\nabla^2 V(x)\,\nabla f(x)\rangle \le \Gamma_2(f,f)(x) + \beta\,\norm{\nabla f(x)}^2\,.
    \end{align*}
    Recall that $\Gamma_2(f,f) = \langle \nabla f,\nabla^2 V\,\nabla f\rangle + \norm{\nabla^2 f}_{\rm HS}^2$.
    Since we have chosen $\nabla f(x) = u$ and $\nabla^2 f(x) = 0$, we finally obtain
    \begin{align*}
        \langle u, \nabla^2 V(x)\,u\rangle \le \beta\,\norm u^2\,.
    \end{align*}
    Since this holds true for all $u,x\in\R^d$, it proves $\nabla^2 V \preceq \beta I$.
\end{proof}

\begin{remark}[Comparison with the local LSI]
    The local gradient-entropy bound~\ref{LGE} can be compared to the following two forms of the local log-Sobolev inequality, which are equivalent to the curvature-dimension condition $\CD(\alpha,\infty)$ (see~\cite[Theorem 5.5.2]{bakry2014analysis}):
    \begin{align*}
        P_t\bigl(\frac{\norm{\nabla f}^2}{f}\bigr)
        &\ge \frac{2\alpha}{1-\exp(-2\alpha t)}\,\{P_t(f\log f) - P_t f \log P_t f\}\,, \\
        \frac{\norm{\nabla P_t f}^2}{P_t f}
        &\le \frac{2\alpha}{\exp(2\alpha t) - 1}\,\{P_t(f\log f) - P_t f \log P_t f\}\,.
    \end{align*}
\end{remark}

We conclude this section with a few speculative comments.

The chain of implications in Theorem~\ref{thm:equiv}---with sharp constants---is, to our knowledge, new.
However, we are unable to close the chain, e.g., by showing that~\ref{CurvUB} suffices to establish~\ref{SHp}.
As discussed in \S\ref{sec:intro}, characterizations of~\ref{CurvUB} would be of great interest toward a synthetic theory of Ricci curvature upper bounds.
In this direction, we remark that if~\ref{SHlog} is suitably generalized to a Riemannian manifold $\mc M$ by replacing the translation $f(\cdot + v)$ with $f\circ \exp v$, where $v$ is a smooth vector field on $\mc M$, then the implication $\text{\ref{SHlog}} \Rightarrow \text{\ref{CurvUB}}$ appears to generalize to yield $\nabla^2 V + \msf{Ric} \le \beta$, essentially because $\Gamma_2(f,f) = \langle \nabla f, (\nabla^2 V + \msf{Ric})\,\nabla f\rangle + \norm{\nabla^2 f}_{\rm HS}^2$.

We also note that another route toward characterizing~\ref{CurvUB} would be to show that it is equivalent to a ``gradient commutation bound''.
For example, let us try to establish~\ref{SHlog} via a semigroup calculation.
It is natural to differentiate $s \mapsto P_s \phi(P_{t-s}(f(\cdot + a_s v)))$ where $\{a_s\}_{s\in [0,t]}$ satisfies $a_0 = 0$ and $a_t = 1$ and we take $\phi \deq \log$.
Upon carrying out the calculation, one sees that the following gradient bound is needed:
\begin{align}\label{eq:proposed_gradient_bd}
    \norm{P_s \nabla f} \overset{?}{\le} \exp(\beta s)\,\norm{\nabla P_s f}\,, \qquad\text{for all}~s \ge 0\,.
\end{align}
We are unsure if~\eqref{eq:proposed_gradient_bd} holds, and in any case, we are unable to prove it.
We leave the further investigation of these questions to future work.

\subsection{Implications for the stationary distribution}\label{sec:implications_stationary}

Since the sharp constants in Theorem~\ref{thm:equiv} tend to finite limits as $t\to\infty$, it immediately furnishes consequences for the stationary distribution $\pi$ of the Langevin diffusion, if it exists.

\begin{cor}\label{cor:stationary}
    Suppose that $\pi\propto\exp(-V)$ defines a probability measure over $\R^d$ and that $V$ satisfies the two-sided curvature bound $-\beta I \preceq \nabla^2 V \preceq \beta I$, where $\beta > 0$.
    Then, the following hold.
    \begin{itemize}
        \item (Gradient-entropy bound,~\ref{eq:LGE_pi}) For all smooth $f : \R^d\to\R_{>0}$,
        \begin{align*}
            \frac{\norm{\E_\pi \nabla f}^2}{\E_\pi f}
            &\le 2\beta \ent_\pi(f)\,.
        \end{align*}
        \item (Shift Harnack inequality) For all $p > 1$, all $v\in\R^d$, and all $f : \R^d\to\R_{>0}$,
        \begin{align*}
            {\{\E_\pi f(\cdot + v)\}}^p \le \exp\Bigl(\frac{\beta p\,\norm v^2}{2\,(p-1)}\Bigr)\,\E_\pi(f^p)\,.
        \end{align*}
        \item (Reverse R\'enyi transport inequality) For all $q > 1$ and all $v\in\R^d$,
        \begin{align*}
            \Ren_q(\pi * \delta_v \mmid \pi)
            &\le \frac{\beta q\,\norm v^2}{2}\,.
        \end{align*}
        \item (Shift log-Harnack inequality) For all $v\in\R^d$ and all $f : \R^d\to\R_{>0}$,
        \begin{align}\label{eq:SHlogpi}
            \E_\pi f(\cdot + v) \le \log \E_\pi(\exp f) + \frac{\beta\,\norm v^2}{2}\,.
        \end{align}
        \item (Reverse KL transport inequality) For all $v\in\R^d$,
        \begin{align}\label{eq:SRT1pi}
            \KL(\pi * \delta_v \mmid \pi)
            &\le \frac{\beta\,\norm v^2}{2}\,.
        \end{align}
    \end{itemize}
\end{cor}
\begin{proof}
    The proofs are immediate by taking $t\to\infty$ in Theorem~\ref{thm:equiv} and applying duality.
\end{proof}

\begin{remark}
    The reverse KL transport inequality~\eqref{eq:SRT1pi} can be established via a simple calculation assuming only~\ref{CurvUB}.\footnote{We also assume, as we do throughout the paper, that there is enough regularity to justify the computations, e.g., $V$ is twice continuously differentiable.}
    Indeed, since $\nabla^2 V \preceq \beta I$, Taylor expansion yields
    \begin{align*}
        \KL(\pi * \delta_v \mmid \pi)
        &= \int \bigl(V(x+v) - V(x)\bigr)\,\pi(\D x)
        \le \int \bigl(\langle \nabla V(x), v\rangle + \frac{\beta\,\norm v^2}{2}\bigr)\,\pi(\D x)
        = \frac{\beta\,\norm v^2}{2}\,,
    \end{align*}
    where we used the identity $\E_\pi \nabla V = 0$.
    By duality (\S\ref{ssec:prelim:duality}), this also yields the shift log-Harnack inequality~\eqref{eq:SHlogpi}.
    However, we are unable to establish the other statements in Corollary~\ref{cor:stationary} under~\ref{CurvUB} alone, rather than the stronger~\ref{CurvBDD}; c.f.\ the discussions in \S\ref{sec:intro} and \S\ref{sec:equiv}.
\end{remark}

To illustrate the information about $\pi$ that such inequalities capture, we explore the consequences of~\eqref{eq:LGE_pi}.
If we take $f$ with $\E_\pi f = 1$ and set $\D\mu = f\,\D\pi$, then it is equivalent to the following inequality: for every regular probability measure $\mu \ll \pi$,
\begin{align}\label{eq:LGE_stationary}
    2\beta\, \KL(\mu \mmid \pi)
    &\ge \Bigl\lVert \E_\pi \nabla \frac{\D\mu}{\D\pi}\Bigr\rVert^2
    = \Bigl\lVert \E_\mu \nabla \log \frac{\D\mu}{\D\pi}\Bigr\rVert^2
    = \norm{\E_\mu \nabla V}^2\,,
\end{align}
where we used the score identity $\E_\mu\nabla \log \mu = 0$.
We now demonstrate how to extract from this formulation
a concentration inequality for $\nabla V$ under $\pi$.
We restate the result for convenience. 

\gradVbd*

\begin{proof}
    Let $e \in\R^d$ be a unit vector, so that $\norm{\E_\mu \nabla V} \ge \abs{\E_\mu\langle e, \nabla V\rangle}$.
    For shorthand, we write $V_e \deq \langle e, \nabla V\rangle$.
    Before proceeding to the proof, we first provide some intuition.
    Recall the Bobkov{--}G\"otze argument~\cite{BobGot1999Transport}: to show that $V_e$ is sub-Gaussian, i.e., to bound the logarithmic moment-generating function $\log \E_\pi\exp(\lambda V_e)$, we first apply the Donsker{--}Varadhan variational principle to argue that $\log \E_\pi \exp(\lambda V_e) = \sup_{\mu\in\mc P(\R^d)}\{\lambda\, \E_\mu V_e - \KL(\mu\mmid \pi)\}$.
    If $V_e$ is $L$-Lipschitz and if $\pi$ satisfies a transport inequality of the form $W_1(\mu,\pi) \le \sqrt{2\alpha^{-1}\,\KL(\mu \mmid \pi)}$, it implies (since $\E_\pi V_e = 0$)
    \begin{align*}
        \log \E_\pi\exp(\lambda V_e)
        &\le \sup_{\mu\in\mc P(\R^d)}\bigl\{ L\lambda\sqrt{2\alpha^{-1}\,\KL(\mu\mmid \pi)} - \KL(\mu\mmid \pi)\bigr\}
        \le \frac{L^2 \lambda^2}{2\alpha}\,.
    \end{align*}
    
    We follow this argument, but instead of assuming that $V_e$ is Lipschitz and that $\pi$ satisfies a transport inequality, we instead use the bound on $\abs{\E_\mu V_e}$ implied by~\eqref{eq:LGE_stationary}.
    This yields
    \begin{align*}
        \log \E_\pi\exp(\lambda V_e)
        &\le \sup_{\mu\in\mc P(\R^d)}\bigl\{\lambda\sqrt{2\beta \,\KL(\mu \mmid \pi)} - \KL(\mu \mmid \pi)\bigr\}
        \le \frac{\beta \lambda^2}{2}\,,
    \end{align*}
    as desired.
\end{proof}

Using covering arguments, Theorem~\ref{thm:gradV_concentration} implies various other high-probability bounds on the size of $\nabla V$ under $\pi$.
To illustrate, we give a bound on $\norm{\nabla V}$.
With Theorem~\ref{thm:gradV_concentration} in hand, the covering-based proof technique is standard and we provide it for completeness.

\begin{cor}\label{cor:gradV_norm}
    Let $\pi\propto \exp(-V)$ be a probability measure on $\R^d$ such that~\eqref{eq:LGE_pi} holds.
    There is a universal constant $C > 0$ such that for all $0 < \delta < 1/2$, with probability at least $1-\delta$ under $\pi$,
    \begin{align*}
        \norm{\nabla V} \le C\,\bigl(\sqrt{\beta d} + \sqrt{\beta \log(1/\delta)}\bigr)\,.
    \end{align*}
\end{cor}
\begin{proof}
    This bound on the logarithmic moment-generating function in Theorem~\ref{thm:gradV_concentration} shows that for every unit vector $e$ and any $\eta > 0$,
    \begin{align*}
        \pi\{V_e \ge \eta\} \le \exp\Bigl( - \frac{\eta^2}{2\beta}\Bigr)
        \,.
    \end{align*}
    To obtain our gradient bound, let $\cN$ be a $\frac{1}{2}$-net of the unit sphere (i.e., $\cN$ has Hausdorff distance at most $\tfrac{1}{2}$ from the sphere). There exists such a set of size $\exp(O(d))$.
    By taking a union bound over vectors in this net, 
    \begin{align*}
        \pi\Bigl\{\sup_{e\in \cN} V_e \ge \eta\Bigr\} \le \exp\Bigl( - \frac{\eta^2}{2\beta} + O(d)\Bigr)\,.
    \end{align*}
    On this event, we have
    \begin{align*}
        \norm{\nabla V}
        = \Bigl\langle \frac{\nabla V}{\norm{\nabla V}}, \nabla V\Bigr\rangle
        \le \inf_{e\in \cN}{\biggl\{ \abs{\langle e, \nabla V\rangle} + \Bigl\lvert\Bigl\langle \frac{\nabla V}{\norm{\nabla V}} - e,\nabla V\Bigr\rangle\Bigr\rvert\biggr\}}
        \le \sup_{e\in \cN}{\abs{\langle e,\nabla V\rangle}} + \frac{1}{2}\,\norm{\nabla V}\,.
    \end{align*}
    Therefore,
    \begin{align*}
        \pi\{\norm{\nabla V} \ge \eta\}
        &\le \pi\Bigl\{\sup_{e\in \cN}{\abs{V_e}} \ge \frac{\eta}{2}\Bigr\}
        \le \exp\Bigl( - \frac{\eta^2}{8\beta} + O(d)\Bigr)\,.
    \end{align*}
    Finally, by choosing $\eta \asymp \sqrt{\beta\,(d+\log(1/\delta))}$, we obtain the statement.
\end{proof}
    
This bound can be compared with~\cite[Corollary 6]{leeshentian2020gradientconcentration} which implies that if $V$ is \emph{convex} and $\beta$-smooth, then with probability at least $1-\delta$,
\begin{align}\label{eq:LST_bd}
    \norm{\nabla V}
    &\le \sqrt{\beta d} + C \sqrt\beta \log(1/\delta)\,.
\end{align}
This concentration bound was used to sharpen the mixing time bounds for the Metropolis-adjusted Langevin algorithm obtained in~\cite{dwivedi2018log, chenetal2020hmc}.

The bound~\eqref{eq:LST_bd} achieves a sharper leading term of $\sqrt{\beta d}$ (in fact, they can replace $\sqrt{\beta d}$ with $\E_\pi\norm{\nabla V}$, which is sharper since $\E_\pi\norm{\nabla V} \le \sqrt{\E_\pi[\norm{\nabla V}^2]} = \sqrt{\E_\pi \Delta V} \le \sqrt{\beta d}$, where the middle equality follows from integration by parts).
On the other hand, Corollary~\ref{cor:gradV_norm} captures the correct tail behavior (sub-Gaussian rather than subexponential). Moreover, whereas~\cite{leeshentian2020gradientconcentration} relied on the Brascamp{--}Lieb inequality~\cite{BraLie1976} and hence on convexity of $V$, our argument shows that it is really the \emph{smoothness} of $V$ (through~\eqref{eq:LGE_pi}) that matters here.

We also mention that concentration of $V$ itself under $\pi$ has been studied in a number of works and in relation to the dimensional improvement of the Brascamp{--}Lieb inequality; see, e.g.,~\cite{Ngu14DimVar, Wang14HeatCap, FraMadWan16InfoContent, BolGenGui18DimImprove, Chewi23EntropicBarrier}.

%% file: sections/app.tex

\section{Deferred details}\label{sec:app}

\subsection{Tightness}\label{app:tightness}

Here we show tightness of the results in \S\ref{sec:disc} and \S\ref{sec:cont}. Specifically, we show that our forward regularity bounds are tight in all parameters in both the discrete-time and continuous-time settings. We do this by explicitly computing these quantities for the semigroup $(P_t)_{t\geq 0}$ corresponding to the Ornstein--Uhlenbeck (OU) process 
\begin{align}
    \D X_t = -L X_t \,\D t + \sqrt{2} \,\D B_t\,,
    \label{eq:ou}
\end{align}
for $L \geq 0$, and the associated discrete-time Markov kernel $\hat{P}_h$, defined as
\begin{align} 
    \hat{P}_h(x,\cdot) = \cN( r x, 2h I)\,,
    \label{eq:ou-disc}
\end{align}
where $r \deq 1 - Lh$. 
For brevity, we compute the relevant quantities only for Dirac initializations since this clearly implies tightness for general initializations. 

\par We remark that since convolution commutes with the OU semigroup (modulo scaling the convolution), tight forward regularity bounds can be recovered from the tight backward regularity bounds in Paper I~\cite{scr1}. This motivates the ensuing calculations and is why we write the (self-contained) calculations in a way that parallels~\cite{scr1}. However, we emphasize that this direct correspondence between forward and backward regularity is specific to the OU process (relying on its commutation properties and explicit solution), and recall that for general semigroups there are fundamental differences between forward and backward regularity (see the discussion in \S\ref{sec:intro}).

\paragraph*{Tightness of the discrete-time regularity.} As in~\cite{scr1}, an explicit computation gives
\begin{align*}
    \delta_x \hat{P}_h^N = \cN\Bigl( r^N x,\, 2h\, \frac{1 - r^{2N}}{1-r^2} \Bigr)\,.
\end{align*}
By Theorem~\ref{thm:renyi_prop},
\begin{align*}
    \Ren_q( \delta_x \hat{P}_h^N \ast \delta_v \mmid \delta_x \hat{P}_h^N)
    =
    \frac{q\,(1 - r^2)}{4h\,(1 - r^{2N})}\, \|v\|^2\,.
\end{align*} 
This exactly matches the upper bound in Theorem~\ref{thm:reg-disc}, establishing its tightness.

\paragraph*{Tightness of the continuous-time regularity.} A standard calculation for the OU process gives the explicit solution
\begin{align*}
    \delta_x P_T = \cN\Bigl( \exp( - L T)\,x,\, \frac{1 - \exp( - 2L T)}{L} \Bigr)\,.
\end{align*}
By Theorem~\ref{thm:renyi_prop},
\begin{align*}
    \Ren_q( \delta_x P_T \ast \delta_v \mmid \delta_x P_T)
    =
    \frac{qL}{ 2\,(1 - \exp( - 2 L T))}\,\|v\|^2\,.
\end{align*} 
This exactly matches the upper bound in Corollary~\ref{cor:reg-cont}, establishing its tightness.

\subsection{Generalized convolution lemma}\label{app:dual-sd}

\begin{proof}[Proof of Lemma~\ref{lem:sd-dual:convolution}]
    The case of positive $z \geq 0$ is~\cite[Lemma 20]{pabi}; consider now $z < 0$, i.e., dual shifted divergences. To simplify notation, re-parameterize $z$ by $-z$ so that we seek to prove
    \begin{align*}
        \Ren_q^{(-z)}\left( \mu \ast \xi \mmid \nu \ast \xi \right)
        \leq
         \Ren_q^{(-z+a)}(\mu \mmid \nu) 
         +
         S_q(\xi,a)
    \end{align*}
    for $z > 0$.
    We start with the easier case $a \le z$.
    Let $v_z$ be the optimal shift for the left-hand side. Observe that $v_{z-a}  \deq  \frac{z-a}{z}\, v_z$ satisfies $\|v_{z-a}\| \leq z-a$ and $\|v_z - v_{z-a}\| \leq a$. Thus
    \begin{align*}
        \Ren_q^{(-z)}( \mu \ast \xi \mmid \nu \ast \xi)
        = 
            \Ren_q( \mu \ast \xi \ast \delta_{v_z} \mmid \nu \ast \xi )
        \leq
            \Ren_q ( \mu \ast \delta_{v_{z-a}} \mmid \nu)
        +
        S_q(\xi, a)
        \leq
        \Ren_q^{(-z+a)}(\mu \mmid \nu) + 
         S_q(\xi, a)\,.
    \end{align*}
    Above, the first and last steps are by definition of the dual shifted divergence, and the middle step is by the composition rule for R\'enyi divergence (see, e.g.,~\cite[Theorem 2.2]{scr1}).

    For the case $a > z$, let $\mu'$ attain the infimum in the definition of $\Ren_q^{(a-z)}(\mu \mmid \nu)$, noting that this is an ordinary shifted divergence as $a-z > 0$.
    We apply the shifted composition rule (Theorem~\ref{thm:scr}) with $\bs\mu$ taken as the joint distribution of $X \sim \mu * \delta_{v_z}$, $X' \sim \mu'$, $Y = X + Z \sim \mu * \delta_{v_z} * \xi$ (where $X - v_z$ and $X'$ are optimally coupled for the $W_\infty$ distance, and $Z \sim \xi$ is independent of $(X,X')$), and $\bs\nu$ taken as the joint distribution of $X \sim \nu$ and $Y = X + Z \sim \nu * \xi$ ($Z\sim \xi$ is independent of $X$). It yields
    \begin{align*}
        \Ren_q^{(-z)}(\mu \ast \xi \mmid \nu \ast \xi)
        &= \Ren_q(\mu \ast \delta_{v_z} * \xi \mmid \nu \ast \xi)
        \le \Ren_q(\mu' \mmid \nu) + S_q\bigl(\xi, W_\infty(\mu * \delta_{v_z}, \mu')\bigr)\,.
    \end{align*}
    Since $\Ren_q(\mu' \mmid \nu) = \Ren_q^{(a-z)}(\mu \mmid \nu)$ and $W_\infty(\mu \ast \delta_{v_z}, \mu') \le W_\infty(\mu \ast \delta_{v_z},\mu) + W_\infty(\mu, \mu') \le z + a-z = a$, this completes the proof.
\end{proof}